\documentclass[12pt,a4paper,reqno]{amsart}
\usepackage[utf8]{inputenc}
\usepackage{graphicx}
\usepackage[linktocpage=true,colorlinks,citecolor=blue,linkcolor=blue,urlcolor=blue,backref=page]{hyperref}
\usepackage{amssymb}
\usepackage{cite}
\usepackage{bbm}
\usepackage{amsmath}
\usepackage{latexsym}
\usepackage{amscd}
\usepackage{tikz}
\usepackage{amsthm}
\usepackage{mathrsfs}
\usepackage{url}
\usepackage[utf8]{inputenc}
\usepackage[english]{babel}
\usepackage{amsfonts}
\usepackage{mathtools}
\usepackage{comment}
\usepackage[autostyle]{csquotes}
\usepackage[colorinlistoftodos]{todonotes}
\usepackage{enumitem}

\numberwithin{equation}{section}
\def\R{\mathbb R}
\def\Z{\mathbb Z}
\def\C{\mathbb C}

\def\N{\mathbb N}
\def\E{\mathbb E}
\def\EC{\mathbb{E}_{\chi}}
\def\ECJ{\EC^{\bm{j}}}

\def\ee{\varepsilon}

\def\sj{\sigma(\bm{j})}
\def\sjr{\sigma^{\text{random}}(\bm{j})}
\def\ejr{\E^{\bm{j}, \text{rand}} }
\def\gcd{\operatorname{gcd}}

\DeclarePairedDelimiter\ceil{\lceil}{\rceil}
\DeclarePairedDelimiter\floor{\lfloor}{\rfloor}

\def\le {\leqslant}
\def\ge {\geqslant}

\newtheorem{theorem}{Theorem}[section]
\newtheorem{lemma}[theorem]{Lemma}
\newtheorem{proposition}[theorem]{Proposition}

\newtheorem{corollary}[theorem]{Corollary}
\newtheorem{conjecture}[theorem]{Conjecture}

\theoremstyle{remark}
\newtheorem{rem}[theorem]{Remark}

\theoremstyle{definition}

\numberwithin{equation}{section}

\theoremstyle{remark}

\usepackage{bm,colonequals,physics}
\renewcommand{\abs}[1]{\lvert #1 \rvert} 
\newcommand{\card}[1]{\lvert #1 \rvert}
 
\newcommand{\defeq}{\colonequals}

\newcommand{\maps}{\colon}

\newcommand{\ol}{\overline}
\def\F{\mathbb F}

\title{Harper's beyond square-root conjecture}
\author{Victor Y. Wang}
\address{Courant Institute, 251 Mercer Street, New York 10012, USA}
\address{IST Austria, Am Campus 1, 3400 Klosterneuburg, Austria}
\address{Institute of Mathematics, Academia Sinica, Taipei 106319, Taiwan}
\email{vywang@alum.mit.edu}
\author{Max Wenqiang Xu}
\address{Courant Institute, 251 Mercer Street, New York 10012, USA}
\email{maxxu1729@gmail.com}
\date{}

\begin{document}

\begin{abstract}
We explain how the (shifted) Ratios Conjecture for $L(s,\chi)$
would extend a randomization argument of Harper from a conductor-limited range to an unlimited range of ``beyond square-root cancellation'' for
character twists of the Liouville function.
As a corollary,
the Liouville function
would have nontrivial cancellation in arithmetic progressions of modulus just exceeding the well-known square-root barrier.
Morally,
the paper
passes from random matrices to random multiplicative functions.
\end{abstract}

\maketitle


\section{Introduction}
Studying character sums is a classical and central topic in number theory.
Recently, 
in \cite{harpertypical}, Harper established a rather striking new phenomenon that the typical character sums have better than square-root cancellation, which can be viewed as a successful de-randomized version of his earlier celebrated result in the random setting \cite{HarperLow}.
Precisely, he \cite[Theorem~1.1]{harpertypical} showed that
if $1\le x\le r$ and $\min(x, r/x) \to +\infty$, then
\[\EC |\sum_{1\le n \le x}  \chi(n)| = o(\sqrt{x}), \]
where $\EC$ denotes $\frac{1}{r-1}\sum_{\chi}$, and the summation is over all Dirichlet characters of a given modulus $r$, where $r$ is a prime for convenience.
Moreover, if $c\in \{\mu,\lambda\}$ is either the \emph{M\"{o}bius function} $\mu$, or the closely related \emph{Liouville function} $\lambda$, he also established that 
\begin{equation}\label{eqn: muchi}
 \EC |\sum_{1\le n \le x} c(n)  \chi(n)| = o(\sqrt{x})
\end{equation}
for the same range of $x$ \cite[Theorem~3]{harpertypical}.
As noted in \cite[paragraph after Theorem~3]{harpertypical},
the same phenomenon could also be proven for the continuous character family $\chi = n^{it}$.
That is, if $1\le x\le T$ and $\min(x, T/x) \to +\infty$, then
\[\frac{1}{T} \int_0^{T} |\sum_{1\le n\le x} c(n)n^{it}| \,dt = o(\sqrt{x}).\]

The restrictions $x\le r$ and $x\le T$ in the above results naturally appeared in Harper's proof, where he needed such a restriction to perform a clever perfect orthogonality trick.
However, he believed that the same sort of result should also hold for a wider range of $x$, even though perfect orthogonality itself no longer holds.
This is Conjecture~\ref{conj: harper}:
\begin{conjecture}[Harper]\label{conj: harper}
Fix $c\in \{\mu,\lambda\}$.
Let $A$ be a fixed positive constant.
Then
\begin{equation}
\label{discrete-case}
\EC |\sum_{1\le n \le x} c(n)\chi(n)| = o(\sqrt{x})
\end{equation}
as $x\to +\infty$, provided $1\le x\le r^{A}$.
Similarly, if $1\le x\le T^{A}$ and $x\to +\infty$, then
\begin{equation}
\label{continuous-case}
\frac{1}{T} \int_0^{T} |\sum_{1\le n\le x} c(n)n^{it}| \,dt = o(\sqrt{x}).
\end{equation}
Moreover, in both \eqref{discrete-case} and \eqref{continuous-case}, the quantitative upper bound $O\Big(\frac{\sqrt{x}}{(\log \log x)^{1/4}}\Big)$ holds.
\end{conjecture}

This conjecture is due to Harper \cite[(1.2)]{harpertypical}, though strictly speaking he only records the conjecture in the M\"{o}bius case $c=\mu$.
(See also \cite{Gor24} for a conjecture for high moments.)
Given $c$, we refer to \eqref{discrete-case} and \eqref{continuous-case} as the \emph{discrete} and \emph{continuous} cases, respectively.
We concentrate on $c=\lambda$ and discrete $\chi$, which allows for the most elegant treatment out of all four possible cases.
We expect that with more technical work, the ideas involved would extend to the other three cases,
and to other families such as quadratic Dirichlet characters (when $c=\mu$).

The following theorem
concerns the family $\mathcal{F}_r$ of $L$-functions $L(s,\chi)$
associated to Dirichlet characters $\chi\bmod{r}$,
and the family $\mathcal{F}^\ast_r\subset \mathcal{F}_r$ of $L$-functions $L(s,\chi)$
associated to primitive characters $\chi\bmod{r}$.
We say that the Generalized Riemann Hypothesis (GRH) holds for a family
if it holds for each member of the family.
We follow the convention that the principal character $\chi_0\bmod{r}$
has $L$-function $L(s,\chi_0) = (1-r^{-s})\zeta(s)$,
whose GRH is equivalent to the ordinary Riemann Hypothesis (RH) for $\zeta(s)$.

\begin{theorem}\label{thm: ratio}
Assume GRH for $\mathcal{F}_r$,
and the Ratios Conjecture \cite[(5.6)]{conrey2008autocorrelation} for $\mathcal{F}^\ast_r$.
Then Conjecture~\ref{conj: harper} is true
in the discrete case \eqref{discrete-case} when $c=\lambda$.
\end{theorem}

Theorem~\ref{thm: ratio} will follow from Theorem~\ref{THM:detailed-implications}.
Conjecture~\ref{CNJ:R22} is the precise form of the Ratios Conjecture that we use.
See \cite{conrey2005integral,conrey2007applications,conrey2008autocorrelation} for details on the general Moments and Ratios Conjectures, and their rich history based on random matrices.

It is natural to hope that GRH alone might suffice,
by adapting techniques of \cite{soundararajan2009moments,harper2013sharp,szabo2024high} from moments to ratios.
See \cite{bui2021ratios,florea2021negative,bui2023negative}, and references within, for progress in this direction.
It would be interesting to pursue this further.
Let us also mention that over function fields $\F_q(t)$,
there is now enough progress on the Ratios Conjecture
that one can likely prove an unconditional version of Theorem~\ref{thm: ratio} in one or more families for all $q$ large enough in terms of $A$.
See \cite{sawin2022square,bergstrom2023hyperelliptic,MPPRW,ratio2024forthcoming}, and references within, for some relevant developments on ratios.

We believe Theorem~\ref{thm: ratio} is interesting and surprising in its own right,
but there is also a significant application
at $A\approx 2$.
Harper \cite{harpertypical} pointed out that by Perron's formula, one can use
the continuous case \eqref{continuous-case} of Conjecture~\ref{conj: harper} for $c=\mu$
to establish nontrivial cancellation of $\mu(n)$ in short intervals breaking the ``square-root barrier''.
Precisely, this means 
\begin{equation}\label{eqn: mobius}
    \sum_{x\le n \le x+y} \mu(n) = o(y)
\end{equation}
\emph{for all} $x$ and $y$ as long as $y\ge \sqrt{x}/W(x)$ for some $W(x)\to+\infty$ as $x\to +\infty$.
It is known that RH implies
\eqref{eqn: mobius}
if $y\ge x^{1/2+\ee}$;
in fact, $y\ge x^{1/2}e^{(\log{x})^{1/2+\ee}}$ suffices, by \cite[Theorem~1]{joni3} of Soundararajan.
The best unconditional result \cite{MT23}, due to 
Matom\"{a}ki and Ter\"{a}v\"{a}inen, gives an exponent around $0.55$.
We also remark that a key feature here is
that we require \eqref{eqn: mobius} to hold
for ``all" short intervals with length beyond certain threshold. If we only require ``almost all", then celebrated work of Matom\"{a}ki--Radziwi\l{}\l{} \cite{MR16} shows that it is sufficient to only require $y \to +\infty$.

The situation for arithmetic progressions $a\bmod{r}$ is similar under GRH, by Ye \cite{ye2014bounding},
but more complicated unconditionally, due to uniformity issues
near the $1$-line.
To our knowledge, there is no unconditional analog of \eqref{eqn: mobius} available when $r=x^{0.01}$,
but results are available for $r=x^{o(1)}$ (see \cite{joni1} for instance).
This is loosely related to the limited range $r \le (\log{x})^A$ in the Siegel--Walfisz Theorem.
However, we note that if one allows to take an average over $r$, then the Bombieri--Vinogradov large sieve inequality is as good as what the GRH gives.
We refer the reader to some recent developments by Maynard \cite{maynard1, maynard2, maynard3} on various special cases where better results can be obtained when averaging $r$ and references therein.
Other related results beyond the square-root range are available in works such as \cite{joni4,joni5,joni6}.

For simplicity and clarity, we conditionally break the ``square-root barrier" in the arithmetic-progression analog of \eqref{eqn: mobius} for $c=\lambda$, as a corollary of Theorem~\ref{thm: ratio}.

\begin{corollary}
\label{main-cor}
Under the same assumptions as in Theorem~\ref{thm: ratio}, we have as $x\to +\infty$
\begin{equation*}
\sum_{\substack{1\le n\le x \\ n\equiv a\bmod{r}}} \lambda(n) = o\Big(\frac{x}{r}\Big)
\end{equation*}
\emph{for all} $a\in \Z$ and primes $r\le W(x)\sqrt{x}$,
for some function $W(x)\to +\infty$ as $x\to +\infty$.
\end{corollary}

\begin{proof}
We first prove this assuming $r\ge x^{0.49}$.
The case $r\mid a$ follows from the prime number theorem
and the complete multiplicativity of $\lambda$.
It remains to treat the case $r\nmid a$, i.e.~$\gcd(a,r)=1$.
In this case, $\bm{1}_{n\equiv a\bmod{r}} = \EC \ol{\chi}(a)\chi(n)$.
Therefore,
\begin{equation*}
\sum_{\substack{1\le n\le x \\ n\equiv a\bmod{r}}} \lambda(n)
= \EC \ol{\chi}(a) \sum_{1\le n\le x} \lambda(n)\chi(n)
\ll \EC \abs{\sum_{1\le n\le x} \lambda(n)\chi(n)}
\ll \frac{\sqrt{x}}{W(x)^{1+\ee}},
\end{equation*}
say, by Theorem~\ref{thm: ratio} with $A=2.05$.
This is $O(\frac{x}{r W(x)^\ee}) = o(\frac{x}{r})$, since $r\le W(x)\sqrt{x}$.

(In fact, by Theorem~\ref{THM:detailed-implications}, we can take $W(x) = (\log\log{x})^b$ for any fixed $b<1/4$.)

On the other hand, the range $r\le x^{0.49}$ can be handled by GRH.
\end{proof}

This application is less interesting in $\F_q[t]$, where stronger techniques are available due to Sawin \cite{sawin2022square}.
We also note that the Ratios Conjecture can be directly used
to estimate averages like $\EC \sum_{1\le n\le x} \lambda(n)\chi(n)$
(via Perron's formula, as in \eqref{perron} of the present paper),
but not to estimate weighted averages like $\EC \ol{\chi}(a) \sum_{1\le n\le x} \lambda(n)\chi(n)$ in general,
except for special values of $a$ like $a=1$.

\begin{rem}
The analog of Corollary~\ref{main-cor}
for the short-interval case \eqref{eqn: mobius}
is likely susceptible to similar methods.
Conditionally on RH and the Ratios Conjecture for the continuous family of $L$-functions $\zeta(s+it_0)$ indexed by $t_0\in \R$, one should be able to
prove \eqref{eqn: mobius} for $y\ge \sqrt{x}/(\log \log x)^{1/4-\ee}$.
It would be interesting to work out the details.
\end{rem}

To understand how universal the phenomena above are, it would be desirable to prove a random matrix analog of Conjecture~\ref{conj: harper}, in the $N\times N$ unitary groups $U(N)$.
In this setting, partial sums $\sum_{1\le n\le x} c(n)\chi(n)$ roughly correspond to characters $\tr(A,\operatorname{Sym}^k\C^N)$, i.e.~Schur polynomials, for $k\ge 0$,
associated to symmetric powers $\operatorname{Sym}^k\C^N$ of the standard representation $\C^N$.
Let $\nu_N$ be the Haar probability measure on $U(N)$.
\begin{conjecture}\label{CNJ:RMT}
Fix $B>0$.
If $0\le k\le BN$ and $k\to +\infty$, then
$$\int_{U(N)} \abs{\tr(A,\operatorname{Sym}^k\C^N)}\, d\nu_N(A) = o(1).$$
\end{conjecture}
Here $k$ is analogous to $\log{x}$ in Conjecture~\ref{conj: harper}, and $N$ is analogous to $\log{r}$.
Note that the range of $B$ is unrestricted, unlike in \cite{NPS}, which studies $\bigwedge^k\C^N$ instead of $\operatorname{Sym}^k\C^N$.
The connection to M\"{o}bius, or $1/L(s,\chi)$, is roughly given by the identity
\begin{equation*}
\frac{1}{\det(1 - e^{1/2-s}A)}
= \sum_{k\ge 0} \tr(A,\operatorname{Sym}^k\C^N) e^{(1/2-s)k}.
\end{equation*}
Also, note that $\int_{U(N)} \abs{\tr(A,\operatorname{Sym}^k\C^N)}^2\, d\nu_N(A) = 1$, since $\operatorname{Sym}^k\C^N$ is irreducible.
So the $o(1)$ in Conjecture~\ref{CNJ:RMT} is the analog of $o(\sqrt{x})$ in Conjecture~\ref{conj: harper}.

We leave this interesting challenge, Conjecture~\ref{CNJ:RMT}, open for now.
A similar conjecture may hold in the symplectic case $A\in USp(2N)$,
or in the orthogonal case $A\in O(N,\R)$ after replacing $\operatorname{Sym}^k\C^N$ with the kernel of the contraction map $\operatorname{Sym}^k\C^N \to \operatorname{Sym}^{k-2}\C^N$.

\subsection*{Strategy}

We sketch the proof ideas.
Our proof is based on Harper's strategy in \cite{harpertypical}.
Roughly speaking, Harper made a remarkable randomization argument to show that when summing over the given family of Dirichlet characters, the typical behavior of the character sums $\sum_{n} \chi(n)$ is very close to the random sums of Steinhaus Random multiplicative functions (RMF) $\sum_{n}f(n)$. The latter is well studied in Harper's earlier paper \cite{HarperLow} and in particular, a conjecture of Helson \cite{Helson} is proved:
\[\E |\sum_{1\le n \le x}f(n) | = o(\sqrt{x}),\]
which shows the ``better than square-root" cancellation phenomenon holds.\footnote{See \cite{Xu, caich2024random, NPS, guzhang} for further developments about this phenomenon and its universality.}
Once the link between the character sums and partial sums of RMF is built, the rest of the proof in \cite{harpertypical} closely follows the work \cite{HarperLow} in the RMF setting.

A key obstruction in Harper's strategy above, that prevented him from establishing Conjecture~\ref{conj: harper} for larger
$x$ is that, in order to show that the character sums typically behave like the random sums of RMF, Harper crucially used the perfect orthogonality of the character sums.
To make use of this, there is a natural restriction on the ranges of the parameters, e.g.~the relation between the moduli $r$ and the length of the sum $x$.
Our innovation is to use a more complex-analytic approach, namely by assuming the Ratios Conjecture to avoid using the orthogonality and thus we can extend the range of the key parameters to establish Conjecture~\ref{conj: harper} conditionally.

However, certain auxiliary estimates and parameters in our work are more delicate than those in Harper's paper \cite{harpertypical}.
Some of the changes required are listed before Proposition~\ref{prop: key} below.
A key ingredient is to replace \emph{perfect} orthogonality with \emph{approximate} orthogonality on average.
This is achieved in a new even-moment estimate, \eqref{INEQ:new-even}, which has a genuine error term, unlike in Harper's work.
To establish this estimate, we must carefully distinguish between $c(n)\chi(n)$ and $\chi(n)$.\footnote{Indeed, \eqref{INEQ:new-even} would be false if we had $c=1$ instead of $c\in \{\mu,\lambda\}$.}
We remark that a main different feature of the two cases is that, in the unweighted case $\sum_{1\le n \le x} \chi(n)$, one needs $r/x \to +\infty$ (see \cite[Theorems~1 and~2]{harpertypical}) in order to get extra cancellation due to the ``Fourier flip"; while in the weighted case $\sum_{1\le n \le x} c(n) \chi(n)$ with $c\in \{\mu,\lambda\}$, we do not expect such a restriction to be required.
The Ratios Conjecture ultimately provides a source of randomness
for the weighted case
but not present for the unweighted case.

Our work can be interpreted as saying that
a \emph{random matrix model} for $\lambda(n)\chi(n)$, namely the Ratios Conjecture,
justifies the \emph{Steinhaus model} for $\lambda(n)\chi(n)$,
for low-moment upper-bound statistics of the partial sums $\sum_{1\le n\le x} \lambda(n)\chi(n)$.

\subsection*{Notation}

Our notation $\ll,\gg,\asymp,O(\cdot),o(\cdot),\floor{\cdot}$ is standard.
We let $P(n)$ denote the maximum prime divisor of $n$.
We let $\bm{1}_E$ denote $1$ if an event $E$ holds, and $0$ otherwise.
The only other convention that deserves comment is the \emph{Steinhaus average} $\E_f A(f)$, or $\E A(f)$ for short, which we always write with the lowercase letter $f$.
Here $f$ denotes a completely multiplicative function.

Usually, $A(f)$ depends only on $f(p)$ for finitely many primes $p$.
In this case, $\E_f A(f)$ denotes the expected value of $A(f)$ when each $f(p)$ is drawn uniformly, independently, from $\{z\in \C: \abs{z}=1\}$.
This is the only case used in Harper's work \cite{harpertypical}.

However, in \S~\ref{SEC:ratios-randomize}, we will allow $A(f)$ to be a Dirichlet series of the form
$$A(f,\bm{z},\bm{s})
= \sum_{(\bm{m},\bm{n})\in \N^k \times \N^l} \frac{a_{\bm{m},\bm{n}}(f)}
{m_1^{z_1}\cdots m_k^{z_k}n_1^{s_1}\cdots n_l^{s_l}},$$
where $a_{\bm{m},\bm{n}}(f)$ depends only on $(\bm{m},\bm{n})$
and on $f(p)$
for $p\mid m_1\cdots m_kn_1\cdots n_l$.
Here $k,l\ge 0$, and $(\bm{z},\bm{s})\in \C^k \times \C^l$.
Then we let
\begin{equation*}
\E_f A(f,\bm{z},\bm{s})
\defeq \sum_{(\bm{m},\bm{n})\in \N^k \times \N^l} \frac{\E_f a_{\bm{m},\bm{n}}(f)}
{m_1^{z_1}\cdots m_k^{z_k}n_1^{s_1}\cdots n_l^{s_l}}.
\end{equation*}
We note that the series $\E_f A(f,\bm{z},\bm{s})$ may converge absolutely
in a strictly larger domain than the series $A(f,\bm{z},\bm{s})$ does for individual $f$.

\subsection*{Acknowledgements}

We thank Paul Bourgade and Kannan Soundararajan for discussions on random matrices and probability,
Alexandra Florea for helpful comments on the Ratios Conjecture,
and Joni Ter\"av\"ainen for providing several references.
We are also grateful to Alexandra Florea, Adam Harper, Joni Ter\"av\"ainen,
and the referee
for helpful comments on earlier drafts.
The first author is supported by the European Union's Horizon~2020 research and innovation program under the Marie Sk\l{}odowska-Curie Grant Agreement No.~101034413.
The second author is supported by a Simons Junior Fellowship from Simons Foundation. 

\section{Character twists behave like random model}
\label{SEC:ratios-randomize}

In this section, we establish that certain mean values of twists behave like a random model (Proposition~\ref{prop: key}).
Instead of using perfect orthogonality, we use the Ratios Conjecture to break through the barrier for the length of the twisted sum.
In order for the Ratios Conjecture to make sense,
we assume GRH for $\mathcal{F}_r$ throughout this section.
(The family $\mathcal{F}_r$ was defined before Theorem~\ref{thm: ratio}.)
Define $\E_{\chi\ne \chi_0} \defeq \frac{1}{r-2} \sum_{\chi\ne \chi_0}$.
It will be sometimes convenient to exclude $\chi_0$ from statements, and sometimes convenient to include it, so both $\E_{\chi\ne \chi_0}$ and $\E_\chi = \frac{1}{r-1} \sum_\chi$ will appear in our work below.

For every completely multiplicative function $f(n)$,
we let $f^\ast(n)\defeq f(n)\cdot \bm{1}_{\gcd(n,r)=1}$
(which is easily checked to be completely multiplicative),
and we formally let
\begin{equation*}
\begin{split}
L^\ast(s,f) &\defeq \sum_{n\ge 1} f^\ast(n) n^{-s}
= \prod_p (1 - f^\ast(p)p^{-s})^{-1}
= \prod_{p\ne r} (1 - f(p)p^{-s})^{-1}, \\
\sum_{n\ge 1} \mu^\ast_f(n) n^{-s}
&\defeq 1/L^\ast(s,f) = \prod_p (1 - f^\ast(p)p^{-s})
= \prod_{p\ne r} (1 - f(p)p^{-s}).
\end{split}
\end{equation*}
Since $f^\ast$ is completely multiplicative, we have $\mu^\ast_f(n) = \mu(n)f^\ast(n)$.
By definition, $\E_f \frac{L^\ast(z_1,f)L^\ast(z_2,\overline{f})}
{L^\ast(s_1,f)L^\ast(s_2,\overline{f})}$ is the meromorphic function obtained by formally expanding
\begin{equation*}
\frac{L^\ast(z_1,f)L^\ast(z_2,\overline{f})}
{L^\ast(s_1,f)L^\ast(s_2,\overline{f})}
= \sum_{m_1,m_2,n_1,n_2\ge 1} \frac{f^\ast(m_1)\ol{f}^\ast(m_2)\mu^\ast_f(n_1)\ol{\mu}^\ast_f(n_2)}
{m_1^{z_1}m_2^{z_2}n_1^{s_1}n_2^{s_2}}
\end{equation*}
as a $4$-variable Dirichlet series, and applying the average $\E_f$ to each coefficient.
Using notation analogous to $G_\zeta(\alpha;\beta;\gamma;\delta)$ from \cite[(5.10)]{conrey2008autocorrelation}, we have
\begin{equation}
\label{Ef-vs-Gzeta}
\E_f \frac{L^\ast(z_1,f)L^\ast(z_2,\overline{f})}
{L^\ast(s_1,f)L^\ast(s_2,\overline{f})}
= G^\ast_\zeta(z_1-\tfrac12; z_2-\tfrac12; s_1-\tfrac12; s_2-\tfrac12),
\end{equation}
where
\begin{equation*}
G^\ast_\zeta(\alpha;\beta;\gamma;\delta)
\defeq \sum_{\substack{m_1,m_2,n_1,n_2\ge 1 \\ m_1n_1 = m_2n_2}}
\frac{\mu(n_1)\mu(n_2)\bm{1}_{r\nmid m_1m_2n_1n_2}}
{m_1^{1/2+\alpha}m_2^{1/2+\beta}n_1^{1/2+\gamma}n_2^{1/2+\delta}}.
\end{equation*}
(The variables $m_1,m_2,n_1,n_2$ in our notation
correspond respectively to
$m,n,h,j$ in the notation of \cite[(5.10)]{conrey2008autocorrelation}.
Also, the Euler factors of $G_\zeta$ and $G^\ast_\zeta$ match at all primes $p\ne r$.)
The series $G^\ast_\zeta$ converges absolutely on the region
$$\Re(\alpha),\Re(\beta),\Re(\gamma),\Re(\delta)>0,$$ for instance,
by a short calculation with the divisor bound.
However, $G^\ast_\zeta$ is meromorphic on a larger region, as is discussed in \cite{conrey2007applications,conrey2008autocorrelation} in great depth.

A special case of the Ratios Conjecture \cite[(5.6)]{conrey2008autocorrelation} for $L(s,\chi)$ is the following:
\begin{conjecture}
\label{CNJ:R22}
Let $\Re(z_1)=\Re(z_2) = \frac12$
and $\Re(s_1)=\Re(s_2)=\frac12+\ee$.
Let $T \defeq \max(\abs{\Im(z_1)},\abs{\Im(z_2)},\abs{\Im(s_1)},\abs{\Im(s_2)})$.
Then for some absolute constant $\omega\in (0,\frac12]$ (independent of $r$ and $\ee$), we have for all sufficiently small $\ee>0$
\begin{equation*}
\E_{\chi\ne \chi_0}
\frac{L(z_1,\chi)L(z_2,\overline{\chi})}{L(s_1,\chi)L(s_2,\overline{\chi})}
= \mathsf{MT}
+ \frac{O_\ee((1+T)^\ee)}{r^\omega},
\end{equation*}
where
\begin{equation*}
\mathsf{MT} \defeq
\E_f \frac{L^\ast(z_1,f)L^\ast(z_2,\overline{f})}
{L^\ast(s_1,f)L^\ast(s_2,\overline{f})}
+ \frac{H}{r^{z_1+z_2-1}}\,
\E_f \frac{L^\ast(1-z_2,f)L^\ast(1-z_1,\ol{f})}
{L^\ast(s_1,f)L^\ast(s_2,\overline{f})},
\end{equation*}
where $H \defeq \frac{(2\pi)^{z_1+z_2}}{\pi^2}\, \Gamma(1-z_1)\Gamma(1-z_2)
\frac{\sum_{a\in \{0,1\}}}{2} \sin(\frac\pi2 (1-z_1+a))\sin(\frac\pi2 (1-z_2+a))$.
\end{conjecture}

\begin{proof}
[Derivation via the Ratios Recipe]
In addition to the general recipe \cite[(5.6)]{conrey2008autocorrelation},
see \cite[\S~4.3]{conrey2005integral} for some details on the family $\{L(s,\chi):\chi\ne \chi_0\}$, noting that $\chi\bmod{r}$ is primitive.
Write $\chi(-1) = (-1)^a$ with $a\in \{0,1\}$.
Let $\tau(\chi) \defeq \sum_{1\le x\le r} \chi(x) e^{2\pi ix/r}$.
It is known that $$L(s,\chi) = w_\chi X_\chi(s) L(1-s,\ol{\chi}),$$
where $w_\chi \defeq \frac{\tau(\chi)}{i^a r^{1/2}}$
and $X_\chi(s) \defeq 2^s \pi^{s-1} r^{1/2-s} \sin(\frac\pi2 (s+a)) \Gamma(1-s)$.
Moreover, by standard properties of Gauss sums, we have $w_{\ol{\chi}} = \ol{w_\chi} = w_\chi^{-1}$.
On multiplying out
\begin{equation*}
\frac{\Big(L(z_1,\chi) + w_\chi X_\chi(z_1) L(1-z_1,\ol{\chi})\Big)\,
\Big(L(z_2,\overline{\chi}) + w_{\ol{\chi}} X_{\ol{\chi}}(z_2) L(1-z_2,\chi)\Big)}
{L(s_1,\chi)L(s_2,\overline{\chi})}
\end{equation*}
and formally averaging over $\chi$,
as described in \cite[(5.6)]{conrey2008autocorrelation},
using the formula \cite[(4.3.4)]{conrey2005integral},
with cancellation over root numbers as in \cite[\S~4.3]{conrey2005integral},
we get the desired Ratios Conjecture.
\end{proof}

The power saving $r^\omega$ in Conjecture~\ref{CNJ:R22},
with an absolute constant $\omega>0$,
is a standard part of the Ratios Conjecture.
Square-root cancellation over families of $L$-functions was conjectured in \cite{conrey2005integral,conrey2008autocorrelation},
though for general families one can only expect a power saving \cite{diaconu2021third}.

The real parts of $z_1,z_2,s_1,s_2$ lie in the standard ranges of the Ratios Conjecture, as specified in \cite[(2.11b)]{conrey2007applications}.
The imaginary parts are trickier to compare with the literature.
To avoid smoothing issues, which are orthogonal to the main point of the paper, we have put $O_\ee((1+T)^\ee)$ in the error term of Conjecture~\ref{CNJ:R22}, which is consistent with the expected admissible set of vertical shifts in \cite[Conjecture~2]{bettin2020averages}.
Morally, $O((1+T)^{O(1)})$ (which is consistent with the more restricted set of vertical shifts in \cite[(2.11c)]{conrey2007applications}) should be enough to say something interesting, but might require smoothing the sums over $n\le x$ in Conjecture~\ref{conj: harper}.
This is similar to the situation of \cite[Theorem~1]{bettin2020averages},
where smoothing would substantially reduce the set of vertical shifts required.

Conjecture~\ref{CNJ:R22} is an asymptotic for $T\le r^{0.99\omega/\ee}$, for any $\ee>0$, whereas for Theorem~\ref{thm: ratio} it would actually suffice to have an asymptotic for $T\le r^{A/2+0.01}$.\footnote{We actually only really apply Conjecture~\ref{CNJ:R22} with $\ee \asymp \omega/A$.}
However, we do not wish to optimize the $T$-aspect at all in this paper.
Similarly, it could be interesting to weaken the saving $r^\omega$ required in the asymptotic,
along the lines of \cite[Conjecture~3.6~(R2)]{BGW2024forthcoming} for instance.
However, our assumption that $\omega$ is independent of $r$ and $\ee$ is very convenient.

The quantity $H=H(z_1,z_2)$ is independent of $r$, and can be bounded as follows.

\begin{lemma}
\label{LEM:H-bound}
The function $H$ is holomorphic for $\Re(z_1),\Re(z_2)<1$.
Moreover, if $\Re(z_1),\Re(z_2)\in \{\frac12-\delta,\frac12+\delta\}$ with $\delta\in [0,\frac12)$,
then $H\ll_\delta (1+\abs{\Im(z_1)})^\delta(1+\abs{\Im(z_2)})^\delta$.
\end{lemma}

\begin{proof}
The first part is clear since $\Gamma$ is holomorphic away from $\Z\setminus \N$.
The second formula follows from Stirling's bound $\Gamma(s) \ll_\delta (1+\abs{\Im(s)})^{\Re(s)-1/2} e^{-\frac{\pi}{2}\abs{\Im(s)}}$, a consequence of \cite[(5.113)]{iwaniec2004analytic}, on $\Re(s)\in \{\frac12-\delta,\frac12+\delta\}$.
Note that $\sin(\frac\pi2 (s+a)) \ll e^{\frac{\pi}{2}\abs{\Im(s)}}$ for $a\in \{0,1\}$,
and the exponential growth here cancels out the exponential decay factor for $\Gamma(s)$.
\end{proof}

The usual Ratios Conjecture implies a \emph{twisted} Ratios Conjecture:
\begin{conjecture}
\label{CNJ:twisted-R2}
Let $\Re(s_1)=\Re(s_2)=\frac12+\ee$ and $1\le m_1,m_2\le r^\hbar$.
Let $T \defeq \max(\abs{\Im(s_1)},\abs{\Im(s_2)})$.
If $\ee,\hbar>0$ are sufficiently small, then for some absolute constant $\omega'\in (0,\frac12]$ (independent of $r,\ee,\hbar$), we have
\begin{equation*}
\E_{\chi\ne \chi_0}
\frac{\chi(m_1)\overline{\chi}(m_2)}{L(s_1,\chi)L(s_2,\overline{\chi})}
= \E_f \frac{f^\ast(m_1)\overline{f}^\ast(m_2)}{L^\ast(s_1,f)L^\ast(s_2,\overline{f})}
+ O_{\ee}((1+T)^\ee r^{-\omega'}).
\end{equation*}
\end{conjecture}

\begin{proof}
[Proof assuming the Ratios Conjecture~\ref{CNJ:R22}]
We first discuss the convergence regions for various quantities that appear in the proof, as preparation for subsequent contour-shifting arguments.
By \eqref{Ef-vs-Gzeta}, we have
\begin{equation}
\label{fourier-flip-term}
\E_f \frac{L^\ast(1-z_2,f)L^\ast(1-z_1,\ol{f})}
{L^\ast(s_1,f)L^\ast(s_2,\overline{f})}
= G^\ast_\zeta(\tfrac12-z_2; \tfrac12-z_1; s_1-\tfrac12; s_2-\tfrac12).
\end{equation}
However, in notation analogous to \cite[(5.12)--(5.14)]{conrey2008autocorrelation},
we have $G^\ast_\zeta(\alpha;\beta;\gamma;\delta) = A^\ast_\zeta\, Y^\ast_U$,
where $A^\ast_\zeta=A^\ast_\zeta(\alpha;\beta;\gamma;\delta)$ is an explicit Euler product known to be absolutely convergent on the region $\abs{\Re(\alpha)},\dots,\abs{\Re(\delta)} < \frac14$ (see \cite[Remark~2.3]{conrey2007applications}),
and where
\begin{equation}
\label{polar-factor-YU}
Y^\ast_U=Y^\ast_U(\alpha;\beta;\gamma;\delta)
\defeq \frac{L(1+\alpha+\beta,\chi_0) L(1+\gamma+\delta,\chi_0)}
{L(1+\alpha+\delta,\chi_0) L(1+\beta+\gamma,\chi_0)}.
\end{equation}
Since $\Re(s_1),\Re(s_2)>\frac12$, it follows that
the main term of Conjecture~\ref{CNJ:R22},
\begin{equation*}
\mathsf{MT} = G^\ast_\zeta(z_1-\tfrac12; z_2-\tfrac12; s_1-\tfrac12; s_2-\tfrac12)
+ \frac{H}{r^{z_1+z_2-1}}\,
G^\ast_\zeta(\tfrac12-z_2; \tfrac12-z_1; s_1-\tfrac12; s_2-\tfrac12),
\end{equation*}
is holomorphic for $\Re(z_1),\Re(z_2) \in (\frac14, \frac34)$,
and that $\E_f \frac{L^\ast(z_1,f)L^\ast(z_2,\overline{f})}{L^\ast(s_1,f)L^\ast(s_2,\overline{f})}$ itself
is holomorphic for $\Re(z_1),\Re(z_2) > \frac12$.
Moreover, these holomorphic functions have polynomial growth in vertical strips.
For $\mathsf{MT}$, the previous two sentences follow from
the integral representation in \cite[\S~6.4, Lemma~6.7]{conrey2008autocorrelation}.
We emphasize that in $\mathsf{MT}$, there is a now-familiar cancellation of poles,
first observed by \cite[\S~2.5]{conrey2005integral} for moments
and by \cite{conrey2007applications,conrey2008autocorrelation} for ratios.

We now do a contour integral to extract $m_1^{-z_1} m_2^{-z_2}$ terms.
Let $g_0(x)$ be a bump function supported on $(-\frac12,\frac12)$, with $g_0(0)=1$.
For each $m\in \N$, let $g_m(x) \defeq g_0(x-m)$.
The Mellin transform $g^\vee_m(z) \defeq \int_0^\infty g_m(x) x^{z-1}\, dx$ satisfies the bound $g^\vee_m(z) \ll_B m^{\Re(z)} \min(1, m/\abs{\Im(z)})^B$, by repeated integration by parts in $x$ if $m/\abs{\Im(z)} \le 1$.
By Mellin inversion, we have
\begin{equation}
\label{f-extract}
\iint_{2-i\infty}^{2+i\infty}
\frac{\prod_{j=1}^{2} g^\vee_{m_j}(z_j)}{(2\pi i)^2}
\E_f \frac{L^\ast(z_1,f)L^\ast(z_2,\overline{f})}
{L^\ast(s_1,f)L^\ast(s_2,\overline{f})} \, dz_1\, dz_2
= \E_f \frac{f^\ast(m_1)\overline{f}^\ast(m_2)}{L^\ast(s_1,f)L^\ast(s_2,\overline{f})},
\end{equation}
with both sides equal to the $m_1^{-z_1}m_2^{-z_2}$ coefficient of
$\E_f \frac{L^\ast(z_1,f)L^\ast(z_2,\overline{f})}{L^\ast(s_1,f)L^\ast(s_2,\overline{f})}$.
Indeed, \eqref{f-extract} is an identity of holomorphic functions on $\Re(s_1),\Re(s_2)>\frac12$, and can thus be checked for $\Re(s_1),\Re(s_2)\ge 2$, say, where all convergence issues become trivial.

The deterministic side requires more care, due to the shape of $\mathsf{MT}$.
By Mellin inversion,
\begin{equation*}
\E_{\chi\ne \chi_0}
\frac{\chi(m_1)\overline{\chi}(m_2)}{L(s_1,\chi)L(s_2,\overline{\chi})}
= \E_{\chi\ne \chi_0} \iint_{2-i\infty}^{2+i\infty}
\frac{\prod_{j=1}^{2} g^\vee_{m_j}(z_j)}{(2\pi i)^2}
\frac{L(z_1,\chi)L(z_2,\overline{\chi})}
{L(s_1,\chi)L(s_2,\overline{\chi})} \, dz_1\, dz_2,
\end{equation*}
since $L(z,\chi)$ converges absolutely on $\Re(z)=2$,
and $g^\vee_m(z)$ decays rapidly in $\abs{\Im(z)}$.
To proceed, we move $\E_{\chi\ne \chi_0}$ inside the integral.
Next, we shift the integral first to $\Re(z_1)=\Re(z_2)=\frac12$,
to apply Conjecture~\ref{CNJ:R22},
then shift the main term $\mathsf{MT}$ to $\Re(z_1)=\Re(z_2)=\frac12+\frac1{50}$, say (this is possible by the first paragraph of the proof). 
This leads to the estimate
\begin{equation}
\begin{split}
\label{INEQ:main-twisted-estimate}
&\E_{\chi\ne \chi_0}
\frac{\chi(m_1)\overline{\chi}(m_2)}{L(s_1,\chi)L(s_2,\overline{\chi})}
- \iint_{2-i\infty}^{2+i\infty}
\frac{\prod_{j=1}^{2} g^\vee_{m_j}(z_j)}{(2\pi i)^2}
\E_f \frac{L^\ast(z_1,f)L^\ast(z_2,\overline{f})}
{L^\ast(s_1,f)L^\ast(s_2,\overline{f})} \, dz_1\, dz_2 \\
&\ll_{B,\ee} \iint_{\frac12+\frac1{50}-i\infty}^{\frac12+\frac1{50}+i\infty}
\frac{(1+T)^\ee
\prod_{j=1}^{2} (1+\abs{\Im(z_j)})^{\frac{1}{50}+\ee}
m_j^{\Re(z_j)} \min(1, \tfrac{m_j}{\abs{\Im(z_j)}})^B
\, \abs{dz_j}}
{\min(r^\omega,r^{\Re(z_1)+\Re(z_2)-1})}
\end{split}
\end{equation}
by \eqref{fourier-flip-term} and \eqref{polar-factor-YU},
because for $\Re(z_1)=\Re(z_2)=\frac12+\frac1{50}$, we have
(by Lemma~\ref{LEM:H-bound})
$$H\ll (1+\abs{\Im(z_1)})^{\frac{1}{50}}(1+\abs{\Im(z_2)})^{\frac{1}{50}}$$ 
and
(by using RH and the estimate $\abs{1-r^{-s}}\asymp 1$ near the $1$-line
to bound the $L(s,\chi_0)^{\pm 1}$ factors in $Y^\ast_U$)
\begin{equation*}
A^\ast_\zeta Y^\ast_U(\tfrac12-z_2; \tfrac12-z_1; s_1-\tfrac12; s_2-\tfrac12)
\ll_\ee (1+\abs{\Im(z_1)})^\ee(1+\abs{\Im(z_2)})^\ee
(1+T)^\ee.
\end{equation*}
We note that the right-hand side of \eqref{INEQ:main-twisted-estimate} takes into account both the error term $(1+T)^\ee/r^\omega$ from Conjecture~\ref{CNJ:R22}, and the second (or ``dual'') term in $\mathsf{MT}$.

Taking $B=2$, and integrating over $z_j$ (the dominant contribution coming from $\abs{\Im(z_j)} \asymp m_j$), we find that the right-hand side of the inequality \eqref{INEQ:main-twisted-estimate} is
\begin{equation*}
\ll \frac{(1+T)^\ee}{\min(r^\omega,r^{\frac1{50}+\frac1{50}})}
\prod_{1\le j\le 2} m_j^{1+\frac1{50}+\ee+\frac12+\frac1{50}}
\le \frac{(1+T)^\ee r^{4\hbar}}{\min(r^\omega,r^{\frac1{25}})}
\le \frac{(1+T)^\ee}{r^{\omega'}},
\end{equation*}
for suitable $\omega'>0$, provided $\ee$ and $\hbar$ are sufficiently small.
Conjecture~\ref{CNJ:twisted-R2} follows, upon plugging \eqref{f-extract} into \eqref{INEQ:main-twisted-estimate}.
\end{proof}

Next, we pass to a version for $\lambda$.
Let $L^\flat(s,\psi) \defeq L^\ast(s,\psi)/L^\ast(2s,\psi^2)$.
By an Euler product calculation,
$1/L^\flat(s,\psi) = \sum_{n\ge 1} n^{-s}\lambda(n)\psi(n)\bm{1}_{\gcd(n,r)=1}$
for any completely multiplicative function $\psi$.
Moreover, for characters $\chi\bmod{r}$, we have $L^\flat(s,\chi) = L(s,\chi)/L(2s,\chi^2)$.
Assume Conjecture~\ref{CNJ:twisted-R2} for the rest of the section.

\begin{proposition}
\label{PROP:lambda-twisted-R2}
Assume Conjecture~\ref{CNJ:twisted-R2}.
Let $\Re(s_1)=\Re(s_2)=\frac12+\ee$ and $1\le m_1,m_2\le r^\hbar$.
Let $T \defeq \max(\abs{\Im(s_1)},\abs{\Im(s_2)})$.
If $\ee,\hbar>0$ are sufficiently small, then for some absolute constant $\eta \in (0,\omega']$ (independent of $r,\ee,\hbar$), we have
\begin{equation*}
\E_{\chi\ne \chi_0} \frac{\chi(m_1)\overline{\chi}(m_2)}{L^\flat(s_1,\chi)L^\flat(s_2,\overline{\chi})}
= \E_f \frac{f^\ast(m_1)\overline{f}^\ast(m_2)}{L^\flat(s_1,f)L^\flat(s_2,\overline{f})}
+ O_\ee((1+T)^\ee r^{-\eta}).
\end{equation*}
\end{proposition}

\begin{proof}

Let $\eta$ be small in terms of $\omega'$.
Let $\ee$ and $\hbar$ be small in terms of $\eta$.
Since $r$ is prime,
there is at most one character $\chi\ne \chi_0$
for which $\chi^2=\chi_0$.
Let $M\ge 1$ be a parameter.
Fix $\kappa>0$ small in terms of $\ee$.
If $\chi^2\ne \chi_0$, then by GRH, $\sum_{d\le N} \chi^2(d)d^{-2it_j} \ll_\ee (1+\abs{t_j})^\kappa r^\kappa \cdot N^{1/2+\kappa}$ for all $N\ge 1$, since the vertically shifted $L$-function $L(s+2t_j,\chi^2)$ has analytic conductor $\ll (1+\abs{t_j})\, r$ in the sense of \cite[Chapter~5]{iwaniec2004analytic}.
By partial summation over $d$, we conclude that
$L(2s_j,\chi^2) - \sum_{d\le M} \frac{\chi^2(d)}{d^{2s_j}}
= \sum_{d>M} \frac{\chi^2(d)}{d^{2s_j}}
\ll_\ee \frac{(1+T)^\kappa r^\kappa}{M^{\Re(2s_j)}}\cdot M^{1/2+\kappa}$ for $j\in \{1,2\}$.
Also, $\sum_{d\le M} \frac{\chi^2(d)}{d^{2s}} \ll_\ee 1$ trivially.
It follows that, under GRH, 
\begin{equation*}
L(2s_1,\chi^2)L(2s_2,\ol{\chi}^2)
- \sum_{d_1,d_2\le M} \frac{\chi(d_1^2)\overline{\chi}(d_2^2)}{d_1^{2s_1}d_2^{2s_2}}
\ll_\ee 1 \cdot \frac{(1+T)^\kappa r^\kappa}{M^{1/2}}
+ \left(\frac{(1+T)^\kappa r^\kappa}{M^{1/2}}\right)^2
\end{equation*}
if $\chi^2\ne \chi_0$,
whereas the left-hand side is trivially $\ll_\ee 1$ if $\chi^2=\chi_0$.
Multiplying both sides by $\chi(m_1)\overline{\chi}(m_2)L(s_1, \chi)^{-1} L(s_2, \ol{\chi})^{-1}$ and using the GRH bound $1/L(s,\chi) \ll_\ee (1+T)^\kappa r^\kappa$, we get
\begin{equation*}
\E_{\chi\ne \chi_0} \frac{\chi(m_1)\overline{\chi}(m_2)}{L^\flat(s_1,\chi)L^\flat(s_2,\overline{\chi})}
- \sum_{d_1,d_2\le M} \E_{\chi\ne \chi_0} \frac{\chi(d_1^2m_1)\overline{\chi}(d_2^2m_2)}{d_1^{2s_1}d_2^{2s_2}L(s_1,\chi)L(s_2,\overline{\chi})}
\ll_\ee \frac{(1+T)^{4\kappa} r^{4\kappa}}{\min(M^{1/2},r)},
\end{equation*}
where the $r$ in $\min(M^{1/2},r)$
comes from the bound $\E_{\chi\ne \chi_0} \bm{1}_{\chi^2=\chi_0} \ll 1/r$.
If $M\le r^{3\eta}$, then on plugging in Conjecture~\ref{CNJ:twisted-R2} (with $6\eta+\hbar$ in place of $\hbar$) and summing over $d_1,d_2\le M$, we get
\begin{equation}
\begin{split}
\label{INEQ:after-removing-large-square-moduli}
&\E_{\chi\ne \chi_0} \frac{\chi(m_1)\overline{\chi}(m_2)}{L^\flat(s_1,\chi)L^\flat(s_2,\overline{\chi})}
- \sum_{d_1,d_2\le M} \E_f \frac{f^\ast(d_1^2m_1)\overline{f}^\ast(d_2^2m_2)}{d_1^{2s_1}d_2^{2s_2}L^\ast(s_1,f)L^\ast(s_2,\overline{f})} \\
&\ll_\ee \frac{(1+T)^{4\kappa} r^{4\kappa}}{M^{1/2}}
+ \frac{(1+T)^\ee}{r^{\omega'}}
\ll \frac{(1+T)^\ee r^\ee}{M^{1/2}}.
\end{split}
\end{equation}

It remains to bound the tail contribution from $\max(d_1,d_2)>M$.
Since $1/L^\flat(s,f) = L^\ast(2s,f^2)/L^\ast(s,f)
= \sum_{d,n\ge 1} f^\ast(d^2)d^{-2s} f^\ast(n)\mu(n)n^{-s}$, we have
\begin{equation*}
\E_f \frac{f^\ast(m_1)\overline{f}^\ast(m_2)}{L^\flat(s_1,f)L^\flat(s_2,\overline{f})}
- \sum_{d_1,d_2\le M} \E_f \frac{f^\ast(d_1^2m_1)\overline{f}^\ast(d_2^2m_2)}{d_1^{2s_1}d_2^{2s_2}L^\ast(s_1,f)L^\ast(s_2,\overline{f})}
\ll \sum_{\substack{d_1,d_2,n_1,n_2\ge 1 \\ d_1^2m_1n_1 = d_2^2m_2n_2}}
\frac{\bm{1}_{\max(d_1,d_2)>M}}{\abs{d_1^{2s_1}d_2^{2s_2}n_1^{s_1}n_2^{s_2}}}
\end{equation*}
by perfect orthogonality over $f$.
By the divisor bound and the inequality $r^{-\hbar}\le m_1/m_2\le r^\hbar$
(which implies that $r^{-\hbar}\le d_1^2n_1/(d_2^2n_2)\le r^\hbar$),
the sum on the right is
\begin{equation*}
\ll_\ee \sum_{d>M} \sum_{n\ge 1}
\frac{(d^2r^\hbar n)^\ee}{(d^2n)^{1/2+\ee} (d^2n/r^\hbar)^{1/2+\ee}}
\ll_\ee \sum_{d>M} \frac{(d^2r^\hbar)^\ee}{(d^4/r^\hbar)^{1/2+\ee}}
\ll \frac{r^\hbar}{M}.
\end{equation*}
Taking $M=r^{3\eta}$ in \eqref{INEQ:after-removing-large-square-moduli},
and recalling that $\ee$ and $\hbar$ are small in terms of $\eta$,
completes the proof of Proposition~\ref{PROP:lambda-twisted-R2}.
\end{proof}

We now prove an unrestricted version of \cite[Lemma~1]{harpertypical}.
As we remarked in the introduction, one may think of \eqref{INEQ:new-even} as an approximate orthogonality relation on average over $n\le x$.
As part of the proof, we will replace $f^\ast$ with $f$, up to a small error term.
\begin{lemma}[Approximate orthogonality and even-moment estimate]
\label{LEM:even-moment-estimate}
Let $c\in \{\mu,\lambda\}$.
Let $\mathcal{P}$ be a finite nonempty set of primes.
Let $\mathcal{Q} \defeq \mathcal{P} \cup \{p^2: p\in \mathcal{P}\}$
and $U\defeq \max(\mathcal{Q})$.
Let $Q(f) \defeq \sum_{q\in \mathcal{Q}} q^{-1/2} a(q)f(q)$ where $a(q)$ are any complex numbers.
Let $k\in \N$.
Assume $U^k\le r^\ee$ and $1\le x\le r^A$, where $A>0$ is fixed.\footnote{We do not assume $xU^k<r$ like Harper does.
On the other hand, our definition of $Q(f)$ agrees with Harper's definition when $f=\chi$.}
Assume Conjecture~\ref{CNJ:twisted-R2}.
Let $\eta$ be as in the statement of Proposition~\ref{PROP:lambda-twisted-R2}.
Then
\begin{equation}\label{INEQ:new-even}
\E_\chi \chi(m_1)\ol{\chi}(m_2) \abs{\sum_{n\le x} c(n) \chi(n)}^2
- \E_f f(m_1)\ol{f}(m_2) \abs{\sum_{n\le x} c(n) f(n)}^2
\ll_A \frac{x}{r^{\eta/2}},
\end{equation}
uniformly for $1\le m_1,m_2\le U^k$, provided $\ee$ is small enough in terms of $A$.\footnote{Of course, we could write $1\le m_1,m_2\le r^\ee$ instead, and thus eliminate the role of $U^k$ in the statement.
However, we prefer to keep it, in order to highlight the connection to \cite[Lemma~1]{harpertypical}.}
Moreover,
\begin{equation}\label{INEQ:old-even}
\E_f \abs{Q(f)}^{2k} \abs{\sum_{n\le x} c(n) f(n)}^2
\ll \sum_{n\le x} \tilde{d}(n) \abs{c(n)}^2
\cdot (k!) \left(2\sum_{q\in \mathcal{Q}} q^{-1} v_q \abs{a(q)}^2\right)^{\!k},
\end{equation}
where $\tilde{d}(n)\defeq \sum_{d|n}\bm{1}_{p\mid d\Rightarrow p\in \mathcal{P}}$,
and $v_q \defeq \bm{1}_{q\in \mathcal{P}} + 6\cdot \bm{1}_{q\notin \mathcal{P}}$.
\end{lemma}


\begin{proof}
If $r > xU^k$, then the left-hand side of \eqref{INEQ:new-even} vanishes (unconditionally), by perfect orthogonality.
Therefore, by taking $r$ large enough, we see that \eqref{INEQ:old-even} is equivalent to \cite[Lemma~1]{harpertypical}.
It remains to prove \eqref{INEQ:new-even} in general, for $U^k\le r^\ee$ and $1\le x\le r^A$.

For convenience, let $L^c(s,\psi) \defeq L^\ast(s,\psi) \bm{1}_{c=\mu} + L^\flat(s,\psi) \bm{1}_{c=\lambda}$.
By Perron's formula in the form of \cite[Theorem~5.2 and Corollary~5.3]{MV2007}, we have
\begin{equation}
\label{perron}
\sum_{n\le x} c(n)\chi(n) - \int_{1+\ee-iT_0}^{1+\ee+iT_0}
\frac{1}{L^c(s,\chi)} \frac{y^s}{s} \frac{ds}{2\pi i}
\ll_\ee \frac{y\log{y}}{T_0} + \frac{(4+y)^{1+\ee}}{T_0}
\ll_\ee \frac{y^{1+\ee}}{T_0},
\end{equation}
where $y \defeq \floor{x}+0.5 \ge 1.5$.
But by contour shifting and GRH, we have
\begin{equation*}
\left(\int_{1+\ee-iT_0}^{1+\ee+iT_0}-\int_{\frac12+\ee-iT_0}^{\frac12+\ee+iT_0}\right) \frac{1}{L^c(s,\chi)} \frac{y^s}{s} \frac{ds}{2\pi i}
\ll_\ee (rT_0)^\ee \frac{y^{1+\ee}}{T_0}.
\end{equation*}
Since $\abs{\sum_{n\le x} c(n) \chi(n)} \ll_\ee r^\ee x^{0.5+\ee}$ by GRH, we thus have, for each $j\in \{0,1\}$,
\begin{equation*}
(-1)^j\sum_{n\le x} c(n)\chi(n)
+ \int_{\frac12+\ee-iT_0}^{\frac12+\ee+iT_0} \frac{1}{L^c(s,\chi)} \frac{y^s}{s} \frac{ds}{2\pi i}
\ll_\ee r^\ee x^{0.5+\ee}\bm{1}_{j=0} + (rT_0)^\ee \frac{y^{1+\ee}}{T_0}.
\end{equation*}
Let $T_0 \defeq x^{(0.5+\eta)/(1-\ee)}$.
Since $\abs{\abs{z}^2 - \abs{w}^2} \le \abs{z^2 - w^2} = \abs{z-w} \abs{z+w}$,
it follows that
\begin{equation*}
\abs{\sum_{n\le x} c(n)\chi(n)}^2
- \abs{\int_{\frac12+\ee-iT_0}^{\frac12+\ee+iT_0} \frac{1}{L^c(s,\chi)} \frac{y^s}{s} \frac{ds}{2\pi i}}^2
\ll_\ee (r^2T_0)^\ee \frac{x^{1.5+2\ee}}{T_0}.
\end{equation*}
Whenever $1\le m_1,m_2\le U^k$, we thus have
\begin{equation}
\begin{split}
\label{INEQ:after-perron-and-R2}
&\E_{\chi\ne \chi_0} \chi(m_1)\ol{\chi}(m_2) \abs{\sum_{n\le x} c(n) \chi(n)}^2
- \iint_{\frac12+\ee-iT_0}^{\frac12+\ee+iT_0} \E_f \frac{f^\ast(m_1)\ol{f}^\ast(m_2) y^{s_1}y^{\ol{s}_2}\,ds_1\,d\ol{s}_2}{L^c(s_1,f)L^c(\ol{s}_2,\ol{f}) s_1\ol{s}_2 \abs{2\pi i}^2} \\
&\ll_\ee (r^2T_0)^\ee \frac{x^{1.5+2\ee}}{T_0}
+ \iint_{\frac12+\ee-iT_0}^{\frac12+\ee+iT_0} \frac{(1+T_0)^\ee}{r^\eta}
y^{\Re(s_1+s_2)} \,\frac{\abs{ds_1\,ds_2}}{\abs{s_1s_2}},
\end{split}
\end{equation}
by Conjecture~\ref{CNJ:twisted-R2} if $c=\mu$,
or by Proposition~\ref{PROP:lambda-twisted-R2} if $c=\lambda$.
In each case, we take $\hbar\defeq \ee$.

Observe that
\begin{equation*}
\E_f \frac{f^\ast(m_1)\ol{f}^\ast(m_2)}{L^c(s_1,f)L^c(\ol{s}_2,\ol{f})}
= \sum_{\substack{n_1,n_2\ge 1 \\ m_1n_1=m_2n_2}}
\frac{c(n_1) c(n_2) \bm{1}_{r\nmid m_1m_2n_1n_2}}{n_1^{s_1} n_2^{\ol{s}_2}},
\end{equation*}
where the right-hand side is absolutely convergent for $\Re(s_1),\Re(s_2)>\frac12$
(allowing us to freely apply Fubini's theorem in the next few lines).
Assume $m_2\ge m_1$.
Then by Perron's formula twice
(again in the form of \cite[Theorem~5.2 and Corollary~5.3]{MV2007}),
first to approximate an integral over $s_2$ as a sum over $n_2\le x$,
and second to approximate an integral over $s_1$ as a sum over $n_1\le x$,
we get
\begin{equation*}
\begin{split}
&\iint_{\frac12+\ee-iT_0}^{\frac12+\ee+iT_0} \E_f \frac{f^\ast(m_1)\ol{f}^\ast(m_2) y^{s_1}y^{\ol{s}_2}\,ds_1\,d\ol{s}_2}{L^c(s_1,f)L^c(\ol{s}_2,\ol{f}) s_1\ol{s}_2 \abs{2\pi i}^2} \\
&= \iint_{\frac12+\ee-iT_0}^{\frac12+\ee+iT_0}
\sum_{n_2\ge 1}
\sum_{\substack{n_1\ge 1 \\ m_1n_1=m_2n_2}}
\frac{c(n_1) c(n_2) \bm{1}_{r\nmid m_1m_2n_1n_2}}{n_1^{s_1} n_2^{\ol{s}_2}}
\frac{y^{s_1}y^{\ol{s}_2}\,d\ol{s}_2\,ds_1}{s_1\ol{s}_2 \abs{2\pi i}^2} \\
&\approx \int_{\frac12+\ee-iT_0}^{\frac12+\ee+iT_0}
\sum_{n_2\le x}
\sum_{\substack{n_1\ge 1 \\ m_1n_1=m_2n_2}}
\frac{c(n_1) c(n_2) \bm{1}_{r\nmid m_1m_2n_1n_2}}{n_1^{s_1}}
\frac{y^{s_1}\,ds_1}{2\pi i s_1} \\
&= \int_{\frac12+\ee-iT_0}^{\frac12+\ee+iT_0}
\sum_{n_1\ge 1}
\sum_{\substack{n_2\le x \\ m_1n_1=m_2n_2}}
\frac{c(n_1) c(n_2) \bm{1}_{r\nmid m_1m_2n_1n_2}}{n_1^{s_1}}
\frac{y^{s_1}\,ds_1}{2\pi i s_1} \\
&\approx \sum_{n_1\le x}
\sum_{\substack{n_2\le x \\ m_1n_1=m_2n_2}}
c(n_1) c(n_2) \bm{1}_{r\nmid m_1m_2n_1n_2}
= \sum_{\substack{1\le n_1,n_2\le x \\ m_1n_1=m_2n_2}}
c(n_1) c(n_2) \bm{1}_{r\nmid m_1m_2n_1n_2},
\end{split}
\end{equation*}
where the error term in the first approximation is (recall $y \defeq \floor{x}+0.5 $)
\begin{equation*}
\begin{split}
&\ll \int_{\frac12+\ee-iT_0}^{\frac12+\ee+iT_0}
\sum_{n_2\ge 1}
\sum_{\substack{n_1\ge 1 \\ m_1n_1=m_2n_2}}
\frac{1}{n_1^{1/2+\ee}}
\left(\frac{y \bm{1}_{n_2\asymp y}}{T_0\abs{y-n_2}}
+ \frac{(4+y)^{1/2+\ee}}{T_0} \frac{1}{n_2^{1/2+\ee}}\right)
\frac{y^{\Re(s_1)}\,\abs{ds_1}}{\abs{s_1}} \\
&\le \int_{\frac12+\ee-iT_0}^{\frac12+\ee+iT_0}
\sum_{n_2\ge 1}
\frac{1}{(m_2n_2/m_1)^{1/2+\ee}}
\left(\frac{y \bm{1}_{n_2\asymp y}}{T_0\abs{y-n_2}}
+ \frac{(4+y)^{1/2+\ee}}{T_0} \frac{1}{n_2^{1/2+\ee}}\right)
\frac{y^{\Re(s_1)}\,\abs{ds_1}}{\abs{s_1}} \\
&\ll_\ee \left(\frac{1}{(m_2y/m_1)^{1/2+\ee}} \frac{y\log{y}}{T_0}
+ \frac{(4+y)^{1/2+\ee}}{T_0} \frac{1}{(m_2/m_1)^{1/2+\ee}}\right)
y^{1/2+\ee} (\log{T_0})
\ll_\ee \frac{y^{1+2\ee}}{T_0} T_0^\ee
\end{split}
\end{equation*}
and the error term in the second approximation is
\begin{equation*}
\begin{split}
&\ll \sum_{n_1\ge 1}
\sum_{\substack{n_2\le x \\ m_1n_1=m_2n_2}}
\left(\frac{y \bm{1}_{n_1\asymp y}}{T_0\abs{y-n_1}}
+ \frac{(4+y)^{1/2+\ee}}{T_0} \frac{1}{n_1^{1/2+\ee}}\right) \\
&\le \sum_{n_1\ge 1} \frac{y \bm{1}_{n_1\asymp y}}{T_0\abs{y-n_1}}
+ \sum_{n_2\le x} \frac{(4+y)^{1/2+\ee}}{T_0} \frac{1}{(m_2n_2/m_1)^{1/2+\ee}}
\ll \frac{y\log{y}}{T_0} + \frac{y}{T_0}.
\end{split}
\end{equation*}
Plugging this into \eqref{INEQ:after-perron-and-R2}, and writing $T_0^{1-\ee}=x^{0.5+\eta}$, gives
\begin{equation*}
\E_{\chi\ne \chi_0} \chi(m_1)\ol{\chi}(m_2) \abs{\sum_{n\le x} c(n) \chi(n)}^2
- \sum_{\substack{1\le n_1,n_2\le x \\ m_1n_1=m_2n_2}}
c(n_1) c(n_2) \bm{1}_{r\nmid m_1m_2n_1n_2}
\ll_\ee \frac{r^{2\ee}x^{1+2\ee}}{x^\eta} + \frac{T_0^{2\ee}x^{1+2\ee}}{r^\eta}.
\end{equation*}

On the other hand, accounting for the missing character $\chi=\chi_0$, we get
\begin{equation}
\begin{split}
\label{principal-grind}
&\frac{r-1}{r-2} \E_\chi \chi(m_1)\ol{\chi}(m_2) \abs{\sum_{n\le x} c(n) \chi(n)}^2
- \E_{\chi\ne \chi_0} \chi(m_1)\ol{\chi}(m_2) \abs{\sum_{n\le x} c(n) \chi(n)}^2 \\
&= \frac{\abs{\sum_{n\le x,\; \gcd(n,r)=1} c(n)}^2 \bm{1}_{r\nmid m_1m_2}}{r-2}
\ll_\ee \frac{r^\ee x^{1+\ee}}{r-2},
\end{split}
\end{equation}
by RH.
Additionally, by perfect orthogonality over $f$, we have
\begin{equation*}
\begin{split}
&\E_f f(m_1)\ol{f}(m_2) \abs{\sum_{n\le x} c(n) f(n)}^2
- \sum_{\substack{1\le n_1,n_2\le x \\ m_1n_1=m_2n_2}}
c(n_1) c(n_2) \bm{1}_{r\nmid m_1m_2n_1n_2} \\
&= \sum_{\substack{1\le n_1,n_2\le x \\ m_1n_1=m_2n_2}}
c(n_1) c(n_2) \bm{1}_{r\mid m_1m_2n_1n_2}
\ll \frac{x}{r},
\end{split}
\end{equation*}
since $m_1,m_2\le U^k\le r^\ee$ are not divisible by $r$.
Multiplying \eqref{principal-grind} by $\frac{r-2}{r-1}$, we find that
\begin{equation*}
\begin{split}
\E_\chi \chi(m_1)\ol{\chi}(m_2) \abs{\sum_{n\le x} c(n) \chi(n)}^2
&\approx \frac{r-2}{r-1} \E_{\chi\ne \chi_0} \chi(m_1)\ol{\chi}(m_2) \abs{\sum_{n\le x} c(n) \chi(n)}^2 \\
&\approx \frac{r-2}{r-1} \sum_{\substack{1\le n_1,n_2\le x \\ m_1n_1=m_2n_2}}
c(n_1) c(n_2) \bm{1}_{r\nmid m_1m_2n_1n_2} \\
&\approx \frac{r-2}{r-1} \E_f f(m_1)\ol{f}(m_2) \abs{\sum_{n\le x} c(n) f(n)}^2 \\
&\approx \E_f f(m_1)\ol{f}(m_2) \abs{\sum_{n\le x} c(n) f(n)}^2,
\end{split}
\end{equation*}
with respective approximation errors
$\ll_\ee \frac{x^{1+\ee}}{r^{1-\ee}}$,
$\ll_\ee \frac{r^{2\ee}x^{1+2\ee}}{x^\eta} + \frac{T_0^{2\ee}x^{1+2\ee}}{r^\eta}$,
$\ll \frac{x}{r}$,
and $\ll \frac{x}{r}$.

We have already justified \eqref{INEQ:new-even} for $r > xU^k$ at the beginning of the proof.
Now assume $r \le xU^k$.
Then $x \ge r/U^k \ge r^{1-\ee}$, so the total error in the last display is $\ll_\ee x/r^{\eta/2}$, provided $\ee$ is small enough in terms of $A$.
So \eqref{INEQ:new-even} holds.
\end{proof}

For the next lemma, let $\Psi(x,y) \defeq \#\{n\le x: p\mid n\Rightarrow p\le y\}$.

\begin{lemma}
[Smooth number bound]
\label{LEM:smooth-number-bound}
Fix $\theta\in (0,\frac12)$.
Suppose $Q\ge 1$ and $\log{Q} \ll (\log{x})^\theta$.
Then $\Psi(x,Q) \le x/e^{(\log{x})^{1-\theta}}$ for all large enough $x$.
\end{lemma}

\begin{proof}
Let $x$ be large.
Increasing $Q$ if necessary, we may assume $\log{Q} \asymp (\log{x})^\theta$.
We have $Q = (\log{x})^a$, where $a = (\log{Q})/(\log\log{x})$
satisfies $1 \le a \le (\log{x})^{1/2}/(2\log\log{x})$.
By Corollary~7.9 of \cite{MV2007}, we have
\begin{equation*}
\Psi(x,Q) \le x^{1-1/a} x^{((\log{a}) + O(1))/(a\log\log{x})}
\le x^{1-1/a} x^{\theta/a}
= x/e^{(1-\theta)(\log{x})/a},
\end{equation*}
because $(\log{a}) + O(1) \le \theta\log\log{x}$.
Since $(\log{x})/a \ge (1-\theta)^{-1}(\log{x})^{1-\theta}$,
it follows that
$\Psi(x,Q)
\le x/e^{(\log{x})^{1-\theta}}$,
as desired.
\end{proof}

In order to proceed, we need an approximation result from \cite{harpertypical}:
\begin{lemma}[Harper {\cite[Approximation Result~1]{harpertypical}}]\label{LEM: APPROX}
    Let $N\in \mathbb{N}$ be large, and $\delta>0$ be small. There exist functions $g: \mathbb{R} \to \mathbb{R}$ (depending on $\delta$) and $g_{N+1}: \mathbb{R} \to \mathbb{R}$ (depending on $\delta$ and $N$) such that, if we define $g_j(x)=g(x-j)$ for all integers $|j|\le N$, we have the following properties: 
    \begin{enumerate}[label=(\roman*)]
        \item $\sum_{|j|\le N}g_j(x) + g_{N+1}(x) = 1$ for all $x\in \mathbb{R}$; 
        \item $g(x)\ge 0$ for all $x \in \mathbb{R}$, and $g(x)\le \delta $ whenever $|x|>1$;
        \item $g_{N+1}(x)\ge 0$ for all $x\in \mathbb{R}$, and $g_{N+1}(x)\le \delta $ whenever $|x|\le N$;
        \item for all $\ell \in \mathbb{N}$ and all $x\in \mathbb{R}$, we have $|\frac{d^{\ell}}{dx^{\ell}} g(x)|\le \frac{1}{\pi(\ell +1 )} \Big(\frac{2\pi}{\delta}\Big)^{\ell +1}$.
    \end{enumerate}
\end{lemma}

We are now prepared to give a variant, Proposition~\ref{prop: key}, of \cite[Proposition~1]{harpertypical},
with the following main differences:
\begin{enumerate}
\item Most importantly, we allow $x\le r^A$ for any fixed $A>0$.
The cost is that we only allow \emph{special} sequences of coefficients $c(n)$.
More precisely, we either take $c(n) = \mu(n)$ for all $n$,
or take $c(n) = \lambda(n)$ for all $n$.

\item We introduce a restriction $P^{400(Y/\delta)^2 \log(N\log{P})} \le r^\ee$,
in order to be able to apply the Ratios Conjecture.
We will actually assume the conditions \eqref{technical-prop-1-condition-1} and \eqref{technical-prop-1-condition-2}, which are nonetheless satisfied for the key choice of parameters in \eqref{key-prop-1-parameters}.

\item The deterministic side is supported on all moduli $n\le x$,
while the random side is supported on $\{n\le x: P(n) > Q\}$ for some parameter $Q\le e^{(\log{x})^{1/3}}$.
The perfect character orthogonality used in \cite[\S~3.2]{harpertypical} (to pass to $P(n)>Q$) is no longer available for large $x$, so it seems more natural for us to pass to $P(n) > Q$ on the random side, not the deterministic side.

The cost of this asymmetry is that we can no longer easily\footnote{Harper notes that $\Psi(x,x^{1/\log\log{x}}) \ll x/(\log{x})^{c\log\log\log{x}}$, but this bound does not fit into the error term of our proposition, in the key setting \eqref{key-prop-1-parameters}.}
take $Q = x^{1/\log\log{x}} = e^{(\log{x})/\log\log{x}}$ like Harper does in \cite[\S~3.2]{harpertypical}.
Moreover, in our \S~\ref{SUBSEC:pass-to-Euler} below, we will need to take a certain smoothing parameter $X$ to be a bit smaller than Harper's choice of $e^{(\log{x})^{1/2}}$ in \cite[\S~3.4]{harpertypical},
since the proof of a certain sieve bound (in Lemma~\ref{LEM:applied-sieve-bounds})
requires $X$ to be a bit smaller than $Q$.

\item We get a worse error term, with $x/(N\log{P})^{10Y}$ in place of $x/(N\log{P})^{Y/\delta^2}$.
Essentially, the effect is that
an error term of the form $x\Big(\frac{2N+2}{(N\log{P})^{1/\delta^2}}\Big)^Y$ in \cite[\S~3.3]{harpertypical}
becomes $x \Big(\frac{2N+2}{(N\log{P})^{10}}\Big)^Y$ in \S~\ref{SUBSEC:pass-to-random} of our work below.
\end{enumerate}

At this point it may help to note that we will eventually apply the proposition with parameter values analogous to those of Harper:
\begin{equation}
\label{key-prop-1-parameters}
\log{P} \asymp (\log{x})^{1/6},
\quad Y \asymp (\log{P})^{1.02},
\quad N \asymp \log\log{P},
\quad \delta \asymp (\log{P})^{-1.3}.
\end{equation}
However, we will state it in slightly greater generality for the reader's convenience.

\begin{proposition}
\label{prop: key}
Assume Conjecture~\ref{CNJ:twisted-R2}.
Let $x,N,\delta^{-1},P>0$ be large real numbers, with $N\in \N$.
Let $c\in \{\mu,\lambda\}$
and $1\le Q\le e^{(\log{x})^{1/3}}$.
Let $g_j\maps \R\to \R$, for $j\in [-N,N+1]$, be functions as in Lemma~\ref{LEM: APPROX}, with associated parameters $N$ and $\delta$.
Let $Y\in \N$.
Fix $A,\ee>0$, suppose $x\le r^A$, and assume that
\begin{align}
\max(P,\delta^{-1},N)^{400(Y/\delta)^2 \log(N\log{P})} &\le r^\ee,
\label{technical-prop-1-condition-1} \\
20Y\log(N\log{P}) &\le (\log{x})^{1/2}.
\label{technical-prop-1-condition-2}
\end{align}
Let $j(1),\dots,j(Y)\in [-N,N+1]$ be indices.
Let $$G_{f,\bm{j}}(\bm{a})\defeq \prod_{1\le i\le Y} g_{j(i)}\Big(\Re \sum_{p\le P} \frac{a_i(p)f(p)}{p^{1/2}} + \frac{a_i(p^2)f(p^2)}{p}\Big),$$
where $a_i(p),a_i(p^2)\in \{z\in \C: \abs{z}\le 1\}$ for all $i$ and $p$.
If $\ee$ is small enough in terms of $A$, then
\begin{equation}
\label{key-prop-goal}
\E_\chi G_{\chi,\bm{j}}(\bm{a}) \abs{\sum_{n\le x} c(n) \chi(n)}^2
= \E_f G_{f,\bm{j}}(\bm{a}) \abs{\sum_{\substack{n\le x: \\ P(n)>Q}} c(n) f(n)}^2
+ O\Big(\frac{x}{(N\log{P})^{10Y}}\Big).
\end{equation}
\end{proposition}

\begin{proof}
Before proceeding, we note that $0\le g_j(y)\le 1$ for all $j\in [-N,N+1]$ and $y\in \R$, by Lemma~\ref{LEM: APPROX} properties~(i)--(iii).
In particular, $0\le G_{f,\bm{j}}(\bm{a})\le 1$.
Also, by properties~(iv) and~(i) in Lemma~\ref{LEM: APPROX},
we have for all $l\ge 0$
\begin{equation}
\label{gj-derivative-bound}
\frac{\abs{g^{(l)}_j(0)}}{l!}
\le \bm{1}_{l=0}
+ \frac{(2N+1) (2\pi/\delta)^{l+1}}{(l+1)!} \bm{1}_{l\ge 1}.
\end{equation}

For the main proof,
we first reduce to the case $Q=0$,
i.e.~we bound the ``missing'' contribution from $P(n)\le Q$ on the right-hand side of \eqref{key-prop-goal}.
Squaring both sides of $$\sum_{\substack{n\le x: \\ P(n)>Q}} c(n) f(n)
= \sum_{n\le x} c(n) f(n) - \sum_{\substack{n\le x: \\ P(n)\le Q}} c(n) f(n),$$
then using the bound $\abs{G_{f,\bm{j}}(\bm{a})} \le 1$
and the triangle inequality, we find that
\begin{equation*}
\E_f G_{f,\bm{j}}(\bm{a}) \abs{\sum_{\substack{n\le x: \\ P(n)>Q}} c(n) f(n)}^2
- \E_f G_{f,\bm{j}}(\bm{a}) \abs{\sum_{n\le x} c(n) f(n)}^2
\ll S_{Q,Q} + S_{Q,x},
\end{equation*}
where
\begin{equation*}
S_{Q,R} \defeq
\E_f \abs{\sum_{\substack{n\le x: \\ P(n)\le Q}} c(n) f(n)}
\abs{\sum_{\substack{n\le x: \\ P(n)\le R}} c(n) f(n)}.
\end{equation*}
By orthogonality over $f$, we have $S_{R,R} \le \Psi(x,R)$.
By the Cauchy--Schwarz inequality, $S_{Q,x} \le S_{Q,Q}^{1/2} S_{x,x}^{1/2} \le \Psi(x,Q)^{1/2} x^{1/2}$.
By Lemma~\ref{LEM:smooth-number-bound} with $\theta=1/3$,
we have $\Psi(x,Q) \ll x/e^{(\log{x})^{1/2}}$, say.
So
\begin{equation*}
S_{Q,Q} + S_{Q,x}
\le \Psi(x,Q) + \Psi(x,Q)^{1/2} x^{1/2}
\ll x/e^{\frac12(\log{x})^{1/2}}.
\end{equation*}
This fits into the error term of \eqref{key-prop-goal}, by \eqref{technical-prop-1-condition-2}.

Therefore, it remains to prove
\begin{equation}
\label{EQN:Q0-goal}
\E_\chi G_{\chi,\bm{j}}(\bm{a}) \abs{\sum_{n\le x} c(n) \chi(n)}^2
= \E_f G_{f,\bm{j}}(\bm{a}) \abs{\sum_{n\le x} c(n) f(n)}^2
+ O\Big(\frac{x}{(N\log{P})^{10Y}}\Big),
\end{equation}
i.e.~the $Q=0$ case.
The estimate \eqref{EQN:Q0-goal} will be proven using Lemma~\ref{LEM:even-moment-estimate}.
The main new difficulty is to account for the error term in \eqref{INEQ:new-even};
the contribution from terms of the form \eqref{INEQ:old-even} will essentially match the error terms Harper already accounted for.
We begin by using Lemma~\ref{LEM: APPROX} as Harper does in \cite[proof of Proposition~1]{harpertypical}.
Let
\begin{equation}
\label{define-S-in-Prop-1}
S \defeq 100Y\floor{(1/\delta)^2 \log(N\log{P})} \ge 100.
\end{equation}
For all $y\in \R$, we then have $g_j(y) = \tilde{g}_j(y) + r_j(y)$,
where $\tilde{g}_j(y) \defeq \sum_{0\le l\le 2S-1} g^{(l)}_j(0) \frac{y^l}{l!}$ and
\begin{equation*}
\abs{r_j(y)} \le \alpha
\frac{N}{\delta} \frac{\abs{2\pi y/\delta}^{2S}}{(2S+1)!},
\end{equation*}
by \cite[first paragraph of the proof of Proposition~1]{harpertypical},
for some absolute constant $\alpha>0$ that we name for later convenience.
The bound \eqref{gj-derivative-bound} implies
\begin{equation*}
\sum_{0\le l\le 2S-1} \frac{\abs{g^{(l)}_j(0)y^l}}{l!}
\le 1 + \sum_{0\le l\le 2S-1} \frac{(2N+1) (2\pi) \abs{2\pi y/\delta}^l}
{\delta (l+1)!}
\ll \frac{N}{\delta} \frac{\max(4S,\abs{2\pi y/\delta})^{2S}}{(2S)!},
\end{equation*}
where in the final step we note that $\frac{\max(4S,\abs{2\pi y/\delta})^l}{(l+1)!} \le \frac12 \frac{\max(4S,\abs{2\pi y/\delta})^{l+1}}{(l+2)!}$ for $0\le l\le 2S-2$,
and that $1 \ll \frac{N}{\delta}$ since $N,\delta^{-1}$ are large.
Let
\begin{equation*}
\widetilde{G}_{f,\bm{j}}(\bm{a})\defeq \prod_{1\le i\le Y} \tilde{g}_{j(i)}\Big(\Re \sum_{p\le P} \frac{a_i(p)f(p)}{p^{1/2}} + \frac{a_i(p^2)f(p^2)}{p}\Big).
\end{equation*}
Now write $\Re z = \frac12(z+\ol{z})$, and expand $\widetilde{G}$ using the definition of each $\tilde{g}_{j(i)}$.
Note that \eqref{technical-prop-1-condition-1} and \eqref{define-S-in-Prop-1} imply $P^{4SY} \le r^\ee$.
By \eqref{INEQ:new-even} with $U=P^{4SY}\le r^\ee$ and $k=1$, we get
\begin{equation}
\label{handle-polynomial-Gtilde}
\E_\chi \widetilde{G}_{\chi,\bm{j}}(\bm{a}) \abs{\sum_{n\le x} c(n) \chi(n)}^2
- \E_f \widetilde{G}_{f,\bm{j}}(\bm{a}) \abs{\sum_{n\le x} c(n) f(n)}^2
\ll O(1)^Y R_0(Y),
\end{equation}
where for any real $t\ge 0$ we define, for later convenience,
\begin{equation}
\label{define-R0t}
R_0(t) \defeq \frac{x}{r^{\eta/2}}
\left(2 + \frac{N}{\delta} \frac{\max(4S,\abs{2\pi \sum_{p\le P} (p^{-1/2}+p^{-1})/\delta})^{2S}}{(2S)!}\right)^{\!t}.
\end{equation}

Next, the bound $\abs{\tilde{g}_j(y)} \le \abs{g_j(y)} + \abs{r_j(y)} \le 1 + \alpha \frac{N}{\delta} \frac{\abs{2\pi y/\delta}^{2S}}{(2S+1)!}$ observed by Harper gives
\begin{equation}
\label{handle-chi-Taylor-remainder}
\E_\chi \abs{\widetilde{G}_{\chi,\bm{j}}(\bm{a})-G_{\chi,\bm{j}}(\bm{a})}
\abs{\sum_{n\le x} c(n) \chi(n)}^2
\le \sum_{1\le i\le Y} \E_\chi H_1(\chi,i)
\end{equation}
by the triangle inequality,
where for any function $f$ we let $H_1(f,i)$ denote the quantity
\begin{equation*}
\abs{\sum_{n\le x} c(n) f(n)}^2 \prod_{1\le l\le i}
\left(\bm{1}_{l<i} + \frac{\alpha N}{\delta} \frac{\abs{2\pi \sum_{p\le P} (p^{-1/2}a_l(p)f(p)+p^{-1}a_l(p^2)f(p^2))/\delta}^{2S}}{(2S+1)!}\right).
\end{equation*}
Similarly,
\begin{equation}
\label{handle-f-Taylor-remainder}
\E_f \abs{\widetilde{G}_{f,\bm{j}}(\bm{a})-G_{f,\bm{j}}(\bm{a})}
\abs{\sum_{n\le x} c(n) f(n)}^2
\le \sum_{1\le i\le Y} \E_f H_1(f,i).
\end{equation}

However, if we expand the product $\prod_{1\le l\le i}$ in $H_1(f,i)$ into monomials of $f(p)$ and $\ol{f}(p)$, then an application of \eqref{INEQ:new-even} with $U=P^{4Si}\le P^{4SY}\le r^\ee$ and $k=1$ shows that
\begin{equation*}
\E_\chi H_1(\chi,i)
- \E_f H_1(f,i)
\ll O(1)^i R_0(i),
\end{equation*}
where we use the inequalities $\bm{1}_{l<i}\le 1$ and $\abs{a_l(\cdot)}\le 1$
to bound the coefficients of the expansion.
Therefore,
\begin{equation}
\label{pass-from-H1chi-to-H1f}
\sum_{1\le i\le Y} \E_\chi H_1(\chi,i)
- \sum_{1\le i\le Y} \E_f H_1(f,i)
\ll \sum_{1\le i\le Y} O(1)^i R_0(i)
\ll O(1)^Y R_0(Y),
\end{equation}
since $R_0(i+1)/R_0(i) \ge 2$.
Next, we bound $\E_f H_1(f,i)$ as Harper does,
using the fact that $\bm{1}_{l<i} = 0$ for $l=i$.
That is,
we first expand the product $\prod_{1\le l\le i}$ in $H_1(f,i)$ as a sum of $2^{i-1}$ terms,
then apply H\"{o}lder's inequality (as Harper does implicitly) in the form
\begin{equation*}
\begin{split}
&\E_f \abs{\sum_{n\le x} c(n) f(n)}^2
\prod_{l\in J} \abs{\sum_{p\le P} (p^{-1/2}a_l(p)f(p)+p^{-1}a_l(p^2)f(p^2))}^{2S} \\
&\le \prod_{l\in J} \left(\E_f \abs{\sum_{n\le x} c(n) f(n)}^2 \abs{\sum_{p\le P} (p^{-1/2}a_l(p)f(p)+p^{-1}a_l(p^2)f(p^2))}^{2S\card{J}}\right)^{\!1/\card{J}}
\end{split}
\end{equation*}
for various sets $J\subseteq \{1,\dots,i\}$ with $i\in J$,
and finally use \eqref{INEQ:old-even},
to obtain the bound
\begin{equation}
\label{harper-H1f-bound}
\sum_{1\le i\le Y} \E_f H_1(f,i)
\ll O(1)^Y H_2,
\end{equation}
where
\begin{equation*}
H_2 \defeq
\sum_{n\le x} \tilde{d}(n) \abs{c(n)}^2
\sum_{1\le j\le i\le Y} \binom{i-1}{j-1} (jS)!
\left(\frac{N}{\delta} \frac{\abs{2\pi/\delta}^{2S}(2\sum_{p\le P} (p^{-1}+6p^{-2}))^S}{(2S+1)!}\right)^{\!j}.
\end{equation*}
The factor $\binom{i-1}{j-1}$ represents the number of sets $J\subseteq \{1,\dots,i\}$ with $i\in J$ and $\card{J}=j$.

Combining \eqref{handle-polynomial-Gtilde}, \eqref{handle-chi-Taylor-remainder}, and \eqref{handle-f-Taylor-remainder},
via the triangle inequality, we get
\begin{equation}
\label{final-key-prop-grind}
\begin{split}
&\E_\chi G_{\chi,\bm{j}}(\bm{a}) \abs{\sum_{n\le x} c(n) \chi(n)}^2
- \E_f G_{f,\bm{j}}(\bm{a}) \abs{\sum_{n\le x} c(n) f(n)}^2 \\
&\ll O(1)^Y R_0(Y)
+ \sum_{1\le i\le Y} \E_\chi H_1(\chi,i)
+ \sum_{1\le i\le Y} \E_f H_1(f,i)
\ll O(1)^Y (R_0(Y) + H_2),
\end{split}
\end{equation}
where in the final step we bound the $H_1$ contributions using \eqref{pass-from-H1chi-to-H1f} and \eqref{harper-H1f-bound}.
To estimate the terms $R_0(Y)$ and $H_2$ in \eqref{final-key-prop-grind}, we will repeatedly use the well-known bounds $n! \ge (n/e)^n$ and $n! \ll n^{1/2} (n/e)^n$ for integers $n\ge 1$.
By \eqref{define-R0t},
\begin{equation*}
R_0(Y) \le \frac{x}{r^{\eta/2}}
\left(2 + \frac{N}{\delta} \frac{(4S)^{2S} + O(P^{1/2}/\delta)^{2S}}{(2S/e)^{2S}}\right)^{\!Y}
\le \frac{O(N/\delta)^Y x}{r^{\eta/2}}
((2e)^{2S} + O(P^{1/2}/\delta)^{2S})^Y,
\end{equation*}
since $N$ is large and $\delta$ is small.
By \eqref{technical-prop-1-condition-1} we have $\max(P,\delta^{-1},N)^{4SY} \le r^\ee$, so this means
\begin{equation*}
O(1)^Y R_0(Y)
\le \frac{x}{r^{\eta/2}}
(PN\delta^{-1})^{4SY}
\le \frac{r^{3\ee} x}{r^{\eta/2}}
\le \frac{x}{r^{5\ee}}
\le \frac{x}{(NP)^{10SY}}
\le \frac{x}{(N\log{P})^{10Y}},
\end{equation*}
because $(NP)^{2SY} \le r^\ee$ and $S\ge 1$.

We now estimate $H_2$ as Harper does,
after summing over $i\in [j,Y]$ using the hockey-stick identity $\sum_{j\le i\le Y} \binom{i-1}{j-1} = \binom{Y}{j}$.
Noting that $\sum_{n\le x} \tilde{d}(n) \le x \prod_{p\le P} (1-p^{-1})^{-1} \ll x \log{P}$ and $\sum_{p\le P} p^{-1} = \log\log{P} + O(1)$, we get
\begin{equation*}
\begin{split}
H_2 &\ll x \log{P}
\sum_{1\le j\le Y} \binom{Y}{j} (jS)^{1/2}(jS/e)^{jS}
\left(\frac{N}{\delta} \frac{\abs{2\pi/\delta}^{2S}(2.1\log\log{P})^S}{(2S/e)^{2S}}\right)^{\!j} \\
&\ll x \log{P}
\sum_{1\le j\le Y} \binom{Y}{j} (jS)^{1/2} \frac{(2.1\pi^2ej)^{jS}}{(S\delta^2)^{jS}}
\left(\frac{N}{\delta} (\log\log{P})^S\right)^{\!j}.
\end{split}
\end{equation*}
Since $j\le Y$, and $S\delta^2 \ge 99Y\log(N\log{P})$ by \eqref{define-S-in-Prop-1}, we get
\begin{equation*}
\begin{split}
O(1)^Y H_2
&\ll O(1)^Y (x \log{P}) 2^Y (YS)^{1/2}
\max_{1\le j\le Y} \left(\frac{N}{\delta} \frac{(\log\log{P})^S}
{(1.7\log(N\log{P}))^S}\right)^{\!j} \\
&\ll O(1)^Y (x\log{P}) S^{1/2}
\left(\frac{N/\delta}{1.7^S} + \left(\frac{N/\delta}{1.7^S}\right)^{\!Y}\right).
\end{split}
\end{equation*}
By Lemma~\ref{LEM:prop-1-numerics} below, $N/\delta, O(1)^Y, S, \log{P}\le 1.1^S$, so the last line is $\ll x (1.1^4/1.7)^S \le x/1.1^S \le x/(N\log{P})^{Y/\delta^2}$, because $99\log{1.1} \ge 1$.
This suffices, because $\delta^{-2} \ge 10$.
\end{proof}

\begin{lemma}
\label{LEM:prop-1-numerics}
Fix $A>0$.
Suppose $P,N,\delta^{-1}>0$ are large.
Let $Y\ge 1$.
Let $S = 100Y\floor{(1/\delta)^2 \log(N\log{P})}$.
Then $N/\delta, O(1)^Y, S, \log{P}\le 1.1^S$.
\end{lemma}

\begin{proof}
This is clear logarithmically:
$$\log{N}+\log(1/\delta), Y\log{O(1)}, \log{S}, \log\log{P}\le S\log{1.1},$$
because $P,N,\delta^{-1},S$ are large.
\end{proof}

While \cite[Proposition~1]{harpertypical} would not suffice for our purposes below, because of its restrictions on $x$,
we note that \cite[Proposition~2]{harpertypical} will still apply unconditionally, without change, because its statement does not involve $x$ at all.

\section{Proof of Theorem~\ref{thm: ratio}}

\subsection{Overview}

In this section, we prove the following result,
which implies Theorem~\ref{thm: ratio}:
\begin{theorem}
\label{THM:detailed-implications}
Each of the following implies the next:
\begin{enumerate}
\item GRH for $\mathcal{F}_r$,
and the Ratios Conjecture~\ref{CNJ:R22} for $L\ol{L}/L\ol{L}$,
hold.

\item GRH for $\mathcal{F}_r$,
and the twisted Ratios Conjecture~\ref{CNJ:twisted-R2} for $\chi\ol{\chi}/L\ol{L}$,
hold.

\item Lemma~\ref{LEM:even-moment-estimate} holds.

\item Proposition~\ref{prop: key} holds.

\item Fix a constant $A>0$.
Let $10\le x\le r^{A}$ and $q\in [0,1]$.
Then $$\EC |\sum_{1\le n \le x} \lambda(n)\chi(n)|^{2q}
\ll \Big(\frac{x}{1+(1-q)(\log\log{x})^{1/2}}\Big)^q.$$
\end{enumerate}
\end{theorem}

We emphasize that the implied constant in (5) may depend on the (fixed) constant $A$.
Also, the Liouville function $\lambda$ has the convenient property that
multiplication by $\lambda$ preserves the Steinhaus distribution.
(This is used in \eqref{eqn: part} below.)
However, following \cite[\S~4.4]{harpertypical}, it should be possible to treat the M\"{o}bius function $\mu$ with more work.

The implications (1)$\Rightarrow$(2),
(2)$\Rightarrow$(3),
(3)$\Rightarrow$(4)
have already been proven in the previous section.
Therefore, it only remains to prove the implication (4)$\Rightarrow$(5).
The key tool is Proposition~\ref{prop: key},
which helps us to pass the study of the deterministic side to the random side.
Let $c(n)\defeq \lambda(n)$.
If $x\ll_A 1$, then (5) is trivial, so we may assume $x$ is large in terms of $A$.
By using H\"{o}lder's inequality, it is sufficient to  establish that for all $q\in [2/3,1]$,
\[\EC |\sum_{1\le n \le x} c(n)\chi(n) |^{2q}
\ll \Big(\frac{x}{1+(1-q)(\log\log{x})^{1/2}}\Big)^q. \]

\subsection{The conditioning argument}
We first do the ``conditioning" argument. More precisely, we use the idea from \cite{harpertypical}, using a partition function to split the character sum into several pieces, based on values taken by certain character sums over primes. 
The proof is almost identical to \S~3.2 in \cite{harpertypical}. The only difference is that we do not restrict ourselves to the integers with an extra condition $P(n)>Q$ here. 
(The condition $P(n)>Q$ will still eventually come up,
once we apply Proposition~\ref{prop: key}.)

In analogy to $S_k(\chi)$ from \cite[\S~3.1]{harpertypical}, we define 
\[S_k(f, c) \defeq \Re \sum_{p\le P} \left( \frac{f(p)c(p)}{p^{1/2+ik/\log^{1.01}P}} +  \frac{f(p)^{2}c(p)^{2}}{2p^{1+ 2ik/\log^{1.01}P}}  \right) = S_k(fc,1) \]
for all $|k|\le M \defeq 2\log^{1.02}P\in \Z$.
(This is a truncated Taylor expansion of $\log{\abs{L(s,fc)^{-1}}}$; cf.~\cite[\S~3.6]{harpertypical}.)
Let $\bm{j}$ be a $2M+1$-vector and $W$ be a function of $\chi$.  
We next apply the approximation result (Lemma~\ref{LEM: APPROX}) to rewrite the sum $\EC |\sum_{1\le n \le x} c(n)\chi(n) |^{2q}$ as 
\begin{align*}
& \EC |\sum_{1\le n \le x} c(n)\chi(n) |^{2q}
\prod_{i=-M}^{i=M} \sum_{j=-N}^{N+1} g_j(S_i(\chi, c)) \\
&= \sum_{-N\le j(-M),\dots, j(M)\le N+1} \EC |\sum_{1\le n \le x} c(n)\chi(n) |^{2q}
\prod_{i=-M}^{M} g_{j(i)}(S_i(\chi, c)) \\
& = \sum_{-N\le j(-M),\dots, j(M)\le N+1} \sigma(\bm{j}) \ECJ |\sum_{1\le n \le x} \chi(n)c(n)|^{2q},
\end{align*}
where we use the notation 
\[
\sigma(\bm{j}) \defeq \EC\prod_{i=-M}^{M} g_{j(i)}(S_i(\chi, c)),
\quad \ECJ W  \defeq \frac{\bm{1}_{\sigma(\bm{j})\ne 0}}{\sigma(\bm{j})}\,
\EC W \prod_{i=-M}^{M} g_{j(i)}(S_i(\chi, c)). \]
(Note that $g_{j(i)} \ge 0$, so if $\sigma(\bm{j})=0$, then $\EC W \prod_{i=-M}^{M} g_{j(i)}(S_i(\chi, c)) = 0$ for all $W$.)
We next apply the H\"older inequality to $\ECJ$ and conclude that 
\begin{equation}\label{eqn: whole}
\EC |\sum_{1\le n \le x} \chi(n)c(n) |^{2q}
\le \sum_{-N\le j(-M),\dots,j(M) \le N+1} \sigma(\bm{j})
\left(\EC^{\bm{j}} | \sum_{1\le n \le x} \chi(n)c(n) |^{2}  \right)^{\!q}.
\end{equation}

\subsection{Passing to the random side}
\label{SUBSEC:pass-to-random}

We next use Proposition~\ref{prop: key} to pass to the random case.
Let $Y \defeq 2M+1 = 1+4\log^{1.02}P$ and $Q=Q(x) \defeq e^{(\log{x})^{1/3}}$.
Let
\begin{equation}
\label{P-X-relationship}
P \defeq \max\{P_0\le \exp((\log{x})^{1/6}): \log^{0.01}P_0\in \Z\},
\end{equation}
in analogy with \cite[\S~3.1]{harpertypical}.
Let
\begin{equation*}
N \defeq \ceil{1.2\log\log{P}},
\quad \delta \defeq (\log{P})^{-1.3},
\end{equation*}
as in \cite[\S~3.6]{harpertypical}.
Then in particular, $P,Y,N,\delta$ satisfy \eqref{key-prop-1-parameters}.
In fact,
\begin{equation*}
\begin{split}
\max(P,\delta^{-1},N)^{400(Y/\delta)^2 \log(N\log{P})}
&\le P^{400 (5\log^{2.32}{P})^2 (2\log\log{P})}
\le e^{(\log{P})^{5.65}}
\le r^\ee, \\
20Y\log(N\log{P})
&\le 100(\log{P})^{1.02}(2\log\log{P})
\le (\log{P})^{1.03}
\le (\log{x})^{1/2},
\end{split}
\end{equation*}
since $5.65/6 < 1$ (and $x$ is large in terms of $A$)
and $1.03/6 < 1/2$.
Thus the requirements in \eqref{technical-prop-1-condition-1} and \eqref{technical-prop-1-condition-2} hold, so we may apply Proposition~\ref{prop: key} to get
\begin{equation}
\begin{split}
\label{eqn: part}
&\EC^{\bm{j}} | \sum_{1\le n \le x} \chi(n)c(n) |^{2} \\
&= \frac{\bm{1}_{\sigma(\bm{j})\ne 0}}{\sigma(\bm{j})}
\left(  \E \prod_{i=-M}^{M} g_{j(i)} (S_i(f,c)) \Big| \sum_{\substack{n\le x \\ P(n)>Q  }} f(n)c(n) \Big|^{2}
+ O\left( \frac{x}{(N \log P)^{10Y}} \right) \right) \\
&= \frac{\bm{1}_{\sigma(\bm{j})\ne 0}}{\sigma(\bm{j})}
\left(  \E \prod_{i=-M}^{M} g_{j(i)} (S_i(f,1)) \Big| \sum_{\substack{n\le x \\ P(n)>Q  }} f(n) \Big|^{2}
+ O\left( \frac{x}{(N \log P)^{10Y}} \right) \right),
\end{split}
\end{equation}
where in the last step we replace $fc$ with $f$,
noting that $S_i(f,c)=S_i(fc,1)$.
We emphasize that $fc$ has the same distribution as $f$,
because $c=\lambda$ is completely multiplicative and $f$ is Steinhaus.

We note that our $Q$ is smaller than Harper's $x^{1/\log\log{x}}$, and that our error term is worse, with $10Y$ instead of $Y/\delta^2$.
Thus, we will focus our exposition on aspects where we lose something compared to Harper, to clarify why we still win overall.

We next notice that $\sum_{\bm{j}}\sigma(\bm{j})=1$ and apply the H\"{o}lder inequality to the $\sigma(\bm{j})$.
This lets us conclude that the contribution from the ``big Oh" term in \eqref{eqn: part} to \eqref{eqn: whole} is at most 
\[\left(\sum_{-N\le j(-M),\dots,j(M) \le N+1} \sigma(\bm{j})
\frac{\bm{1}_{\sigma(\bm{j})\ne 0}}{\sigma(\bm{j})} \frac{x}{(N\log P)^{10Y}} \right)^q
\le \left(\frac{x (2N+2)^Y}{(N\log{P})^{10Y}}\right)^q
\ll \Big(\frac{x}{\log{P}}\Big)^q,\]
which is $\ll (\frac{x}{1+(1-q)(\log\log{P})^{1/2}})^q$.
Thus, we can ignore the ``big Oh" term in \eqref{eqn: part}. 

We next define 
\[\sigma^{\text{random}}(\bm{j}) \defeq \E \prod_{i = -M}^{M} g_{j(i)}(S_i(f, 1))  \]
for all $(2M+1)$-vectors $\bm{j}$ where $f$ is a Steinhaus random multiplicative function.
We then use \cite[Proposition~2]{harpertypical} to get that 
$\sigma(\bm{j})^{1-q} \ll \sigma^{\text{random}}(\bm{j})^{1-q} 
+ \Big(\frac{1}{(N\log P)^{(2M+1)(1/\delta)^{2}}}\Big)^{1-q}$. This implies that 
\[\sum_{-N\le j(-M),..,j(M)\le N+1} \sj \left(\frac{\bm{1}_{\sj\ne 0}}{\sj}\,
\E  \prod_{i=-M}^{M} g_{j(i)} (S_i(f, 1)) \Big| \sum_{\substack{n\le x \\ P(n)>Q  }} f(n) \Big|^{2} \right)^{q}
\ll T_1 + T_2,\]
where
\begin{equation*}
\begin{split}
T_1 &\defeq \sum_{\bm{j}} \sjr \left(\frac{\bm{1}_{\sjr\ne 0}}{\sjr}\,
\E  \prod_{i=-M}^{M} g_{j(i)} (S_i(f, 1)) \Big| \sum_{\substack{n\le x \\ P(n)>Q  }} f(n) \Big|^{2} \right)^{q}, \\
T_2 &\defeq \Big(\frac{1}{(N\log P)^{(2M+1)(1/\delta)^{2}}}\Big)^{1-q} \sum_{\bm{j}} \left(  \E  \prod_{i=-M}^{M} g_{j(i)} (S_i(f, 1)) \Big| \sum_{\substack{n\le x \\ P(n)>Q  }} f(n) \Big|^{2} \right)^{q},
\end{split}
\end{equation*}
where $\sum_{\bm{j}} \defeq \sum_{-N\le j(-M),...,j(M)\le N+1}$.
To deal with
$T_2$, we apply H\"older's inequality to the sum over $\bm{j}$ and use the fact that $g_j$ form a partition of unity. This gives us
\[
\begin{split}
T_2  & \ll  \Big(\frac{1}{(N\log P)^{(2M+1)(1/\delta)^{2}}}\Big)^{1-q} ((2N+2)^{2M+1})^{1-q} \left( \sum_{\bm{j}}  \E  \prod_{i=-M}^{M} g_{j(i)} (S_i(f, 1)) \Big| \sum_{\substack{n\le x \\ P(n)>Q  }} f(n) \Big|^{2} \right)^{q}  \\
& = \left( \Big( \frac{2N+2}{(N\log P)^{(1/\delta)^{2}}} \Big)^{2M+1} \right)^{1-q} \left( \E \Big|\sum_{\substack{n\le x \\ P(n)>Q  }} f(n) \Big|^{2} \right)^{q}
\le (\log{P})^{-(1-q)} x^q,
\end{split}\]
which is again $\ll (\frac{x}{1+(1-q)(\log\log{P})^{1/2}})^q$.
In summary, we can ignore the contribution from $T_2$.
If we define the notation $\E^{\bm{j}, \text{rand}} W
\defeq \frac{\bm{1}_{\sjr\ne 0}}{\sjr}\,
\E W \prod_{i=-M}^{M} g_{j(i)} (S_i(f, 1))$ for all random variables $W$,
then by \eqref{eqn: whole} and \eqref{eqn: part}, it suffices to show that  
\begin{equation}
\label{purely-rmf-but-with-Q}
T_1 =
\sum_{\bm{j}}
\sjr \left(\ejr \Big| \sum_{\substack{n\le x \\ P(n)>Q  }} f(n) \Big|^{2}  \right)^{q}
\ll \Big(\frac{x}{1+(1-q)(\log\log{P})^{1/2}}\Big)^q,
\end{equation}
in order to complete the proof of Theorem~\ref{THM:detailed-implications}.
The point is that from now on we only need to focus on
$T_1$,
which is purely about random multiplicative functions.

\subsection{Passing to Euler products}
\label{SUBSEC:pass-to-Euler}
Recall that $Q(x)= e^{(\log{x})^{1/3}}$. 
Let $X = e^{(\log{x})^{1/100}}$.
In the corresponding section \cite[\S~3.4]{harpertypical},
Harper uses a standard sieve bound \cite[Theorem~3.6]{MV2007} several times in the course of the argument.
Since our parameters are slightly different than Harper's, it is worth spelling these out more explicitly.

\begin{lemma}
[Applied sieve bounds]
\label{LEM:applied-sieve-bounds}
We have
\begin{equation*}
\sum_{\substack{Q < m \le x \\ p\mid m\Rightarrow p>P}} \frac{X}{m}
\int_{m}^{(1+1/X)m} \left(\frac{x}{m} - \frac{x}{t} + 1\right) \, dt
\ll \frac{x}{\log{P}}.
\end{equation*}
Also, for $t\ge Q$, we have
\begin{equation*}
\sum_{\substack{t/(1+1/X) < m \le t \\ p\mid m\Rightarrow p>P}} \frac{X}{m}
\ll (\log{P})^{-1}.
\end{equation*}
\end{lemma}

\begin{proof}
In the first display, we have
$\frac{x}{m} - \frac{x}{t} + 1 = \frac{(t-m)x}{mt} + 1 \ll \frac{x}{mX} + 1$.
Therefore, we only need to show that 
\[\sum_{\substack{Q < m \le x\\ p\mid m\Rightarrow p>P}} \left(1+ \frac{x}{mX}\right) \ll  \frac{x}{\log P}. \]
Notice that sum involving $x/Xm$ is crudely bounded by $\ll x\log x/X$ which is more than acceptable by our choice of $X$ and $P$.
For the other term, we just use the upper bound sieve \cite[Theorem~3.6]{MV2007}, noting here $P \ll \sqrt{x-Q}$ by \eqref{key-prop-1-parameters}, to conclude that this is $\ll x \prod_{p\le P}(1-p^{-1}) \ll x/\log P$.

We now turn to the second display. We have the bound
\[\sum_{\substack{t/(1+1/X) < m \le t \\ p\mid m\Rightarrow p>P}} \frac{X}{m} \ll \frac{X}{t}\sum_{\substack{t/(1+1/X) < m \le t \\ p\mid m\Rightarrow p>P}} 1.\]
By using the upper bound sieve again,
noting that $P \ll \sqrt{t - t/(1+1/X)}$ by \eqref{key-prop-1-parameters},
the inner sum is $\ll \frac{t}{X} \prod_{p\le P}(1-p^{-1})\ll \frac{t}{X\log P}$ and the conclusion follows. 
\end{proof}

We now return to our goal, \eqref{purely-rmf-but-with-Q}.
Expanding the square in $T_1$ and noticing that $g_{j(i)}(S_i(f, 1)) $ only depends on $f(p)$ for $p\le P$,
we get
\begin{equation}\label{eqn: euler}
T_1
=   \sum_{\bm{j}}\sjr \left(\ejr \sum_{\substack{m\le x\\ P(m)>Q\\ p\mid m \Rightarrow p>P}} \Big| \sum_{\substack{n\le x/m \\ n~\text{is $P$-smooth} }} f(n) \Big|^{2}  \right)^{q} .  
\end{equation}
Note that we have used $P<Q$ above. The next step is to replace the discrete sum with a smooth version. 
Replacing the $P(m)>Q$ by the weaker condition $m>Q$, we find that the above bracket term, $(\ejr \sum \cdots)^q$, in \eqref{eqn: euler} is
\begin{equation}\label{eqn: 1}
\begin{split}
      \ll & \left(\ejr \sum_{\substack{Q<m\le x\\  p|m \Rightarrow p>P}} \frac{X}{m} \int_{m}^{m(1+1/X)} \Big| \sum_{\substack{n\le x/t \\ n~\text{is $P$-smooth} }} f(n) \Big|^{2}  dt \right)^{q}   \\
& + \left(\ejr \sum_{\substack{Q<m\le x\\  p|m \Rightarrow p>P}} \frac{X}{m} \int_{m}^{m(1+1/X)} \Big| \sum_{\substack{x/t < n\le x/m \\ n~\text{is $P$-smooth} }} f(n) \Big|^{2}  dt \right)^{q}. 
\end{split} 
\end{equation} 

We first estimate the contribution of the second term in \eqref{eqn: 1} to \eqref{eqn: euler}, which is, via H\"older's inequality (using $\sum_{\bm{j}} \sjr = 1$),
\[\begin{split}
& \ll  \left(\sum_{\bm{j}} \sjr \ejr \sum_{\substack{Q<m\le x\\  p|m \Rightarrow p>P}} \frac{X}{m} \int_{m}^{m(1+1/X)} \Big| \sum_{\substack{x/t < n\le x/m \\ n~\text{is $P$-smooth} }} f(n) \Big|^{2}  dt \right)^{q} \\ 
& = \left(\E \sum_{\substack{Q<m\le x\\  p|m \Rightarrow p>P}} \frac{X}{m} \int_{m}^{m(1+1/X)} \Big| \sum_{\substack{x/t < n\le x/m \\ n~\text{is $P$-smooth} }} f(n) \Big|^{2}  dt \right)^{q},
\end{split}
\]
by the definition of $\ejr$.
Next we use the orthogonality to expand the square and the above is $\ll (\frac{x}{\log P})^{q}$ by applying Lemma~\ref{LEM:applied-sieve-bounds}.
This is $\ll (\frac{x}{1+(1-q)(\log\log{P})^{1/2}})^q$.

We next estimate the contribution from the first term in \eqref{eqn: 1} to \eqref{eqn: euler}. 
This is
\[
\begin{split}
    & \ll \sum_{\bm{j}} \sjr \left(\ejr  \int_{Q}^{x} \Big| \sum_{\substack{n\le x/t \\ n~\text{is $P$-smooth} }} f(n) \Big|^{2}  \sum_{\substack{t/(1+1/X)<m\le t\\  p|m \Rightarrow p>P}} \frac{X}{m}dt \right)^{q} \\
    & \ll  \sum_{\bm{j}} \sjr \Big(\frac{1}{\log P}\Big)^{q} \left(\ejr  \int_{Q}^{x} \Big| \sum_{\substack{n\le x/t \\ n~\text{is $P$-smooth} }} f(n) \Big|^{2}  dt \right)^{q}\\
    & =  \sum_{\bm{j}} \sjr \Big(\frac{x}{\log P}\Big)^{q} \left(\ejr  \int_{1}^{x/Q} \Big| \sum_{\substack{n\le z\\ n~\text{is $P$-smooth} }} f(n) \Big|^{2}  \frac{dz}{z^{2}} \right)^{q},
\end{split}
\]
where in the second line we applied Lemma~\ref{LEM:applied-sieve-bounds} and in the last line we used $z \defeq x/t$.
Then we can apply \cite[Harmonic Analysis Result 1]{harpertypical} to bound the last line by
\begin{equation}
\label{harper-(3.4)}
\ll \Big(\frac{x}{\log{P}}\Big)^q
\sum_{\bm{j}} \sjr \Big(\ejr \int_\R \frac{\abs{F_P^{\text{rand}}(1/2+it)}^2}
{\abs{1/2+it}^2} \,dt\Big)^q,
\end{equation}
where
\begin{equation}
\label{compatible-euler-product}
F_P^{\text{rand}}(s) \defeq \prod_{p\le P} (1 - \frac{f(p)}{p^s})^{-1}.
\end{equation}
Note that $F_P^{\text{rand}}(s)$ is a truncated version of $L(s,f)$.

The quantity in \eqref{harper-(3.4)} is of the form $(\frac{x}{\log{P}})^q\, \mathcal{R}_1$.
Similarly, the quantity in \cite[(3.4)]{harpertypical} is of the form $(\frac{x}{\log{P}})^q\, \mathcal{R}_2$.
The relationship \eqref{P-X-relationship} between $P$ and $x$ for us is in general different than in \cite{harpertypical}, if $\log_r{x}$ is large enough.\footnote{If $r\ge x^2$, then our definition of $P=P(x)$ in \eqref{P-X-relationship} agrees with
Harper's definition of $P=P(x,r)$ in \cite[\S~3.1]{harpertypical}.
In addition, \eqref{harper-(3.4)} only involves $x,P,N,\delta$, not $r$.
Therefore, one way to proceed would be to invoke Harper's work on \cite[(3.4)]{harpertypical} for primes $r\ge x^2$.
But that may be more confusing than necessary, so we instead proceed by focusing only on the parameters $P,N,\delta$, independently of $x$ and $r$.}
However, because our choice of $N,\delta$ (in terms of $P$) is identical to Harper's,
and our Euler product $F_P^{\text{rand}}(s)$ is identical by \eqref{compatible-euler-product},
we have $\mathcal{R}_1 = \mathcal{R}_2$ as functions of $P$.
Moreover, we chose $P$ so that $\log^{0.01}{P}\in \N$, as in \cite[\S~3]{harpertypical}.
Therefore, by \cite[(3.5) and the paragraph before it]{harpertypical}, we have
\begin{equation}
\begin{split}
\label{endgame}
\mathcal{R}_1
&\defeq \sum_{\bm{j}} \sjr \Big(\ejr \int_\R \frac{\abs{F_P^{\text{rand}}(1/2+it)}^2}
{\abs{1/2+it}^2} \,dt\Big)^q \\
&\ll \Big(\frac{\log{P}}{1+(1-q)(\log\log{P})^{1/2}}\Big)^q.
\end{split}
\end{equation}
We emphasize that \eqref{endgame} really only involves $P,N,\delta$, not $x,r$.
The inequality \eqref{endgame} holds in general under the conditions
(as discussed in \cite[\S~3.1, \S~3.6]{harpertypical})
\begin{equation*}
\log^{0.01}{P}\in \N,
\quad N\ge 1.2\log\log{P},
\quad \delta\le \frac{1}{(\log{P})^{1.2}(\log\log{P})^{0.5}N}.
\end{equation*}
Plugging \eqref{endgame} into \eqref{harper-(3.4)},
and recalling our bounds for $T_1$ based on \eqref{eqn: euler} and \eqref{eqn: 1},
we conclude that \eqref{purely-rmf-but-with-Q} holds.
This completes the proof of Theorem~\ref{THM:detailed-implications}.

\bibliographystyle{plain}
\bibliography{main}{}

\end{document}